\documentclass[a4article,12pt]{amsart}
\usepackage[left=2.50cm, right=2.50cm]{geometry}

\usepackage[mathscr]{eucal}
\usepackage{mathpazo}

\allowdisplaybreaks

\usepackage{amsmath,amsthm,amssymb,amsfonts,amscd}
\usepackage{hyperref,color}
\hypersetup{colorlinks=true,urlcolor=black,linkcolor=black,citecolor=black}

\newtheorem{theorem}{Theorem}[section]
\newtheorem{example}[section]{Example}
\newtheorem*{theorem*}{Theorem}
\newtheorem{lemma}[theorem]{Lemma}
\newtheorem{proposition}[theorem]{Proposition}
\newtheorem{definition}[theorem]{Definition}
\newtheorem{notation}[theorem]{Notation}
\newtheorem{corollary}[theorem]{Corollary}

\newtheorem{remark}[theorem]{Remark}
\newtheorem{cohftaxiom}{CohFT}
\newtheorem{unitaxiom}{U}

\newcommand{\CC}{{\mathbb{C}}}

\newcommand{\HH}{{\mathbb{H}}}
\newcommand{\PP}{{\mathbb{P}}}
\newcommand{\QQ}{{\mathbb{Q}}}

\newcommand{\ZZ}{{\mathbb{Z}}}
\newcommand{\LL}{{\mathbb{L}}}

\newcommand{\NN}{{\mathbb{N}}}

\newcommand{\Fix}{\mathrm{Fix}}

\newcommand{\Aut}{\mathrm{Aut}}

\newcommand{\bx}{{\bf x}}

\newcommand{\bt}{{\bf t}}

\newcommand{\id}{{\rm id}}
\renewcommand{\Im}{{\rm Im}}

\def\SL2C{ {\rm SL(2,\mathbb C)} }
\newcommand{\M}[2]{ { \overline{\mathcal M}_{#1, #2} } }

\def\A{{\mathcal A}}

\def\C{{\mathcal C}}

\def\E{{\mathcal E}}
\def\F{{\mathcal F}}

\def\H{{\mathcal H}}

\def\L{{\mathcal L}}
\def\calM{{\mathcal M}}

\def\X{{\mathcal X}}
\def\Z{{\mathcal Z}}
\def\W{{\mathcal W}}

\def\rx{{\mathrm x}}
\def\ry{{\mathrm y}}
\def\rz{{\mathrm z}}
\def\rw{{\mathrm w}}

\def\p{\partial }

\begin{document}
\title{6--dimensional FJRW theories of the simple--elliptic singularities}
\begin{abstract}
We give explicitly in the closed formulae the genus zero primary potentials of the three 6--dimensional FJRW theories of the simple--elliptic singularity $\tilde E_7$ with the non--maximal symmetry groups. For each of these FJRW theories we establish the CY/LG correspondence to the Gromov--Witten theory of the elliptic orbifold $[\E/ (\ZZ/2\ZZ)]$ --- the orbifold quotient of the elliptic curve by the hyperelliptic involution. Namely, we give explicitly the Givental's group elements, whose actions on the partition function of the Gromov--Witten theory of $[\E/ (\ZZ/2\ZZ)]$ give up to a linear change of the variables the partition functions of the FJRW theories mentioned. We keep track of the linear changes of the variables needed.
We show that using only the axioms of Fan--Jarvis--Ruan, the genus zero potential can only be reconstructed up to a scaling.
\end{abstract}
\date{\today}

\author{Alexey Basalaev}
\address{Universit\"at Mannheim, Lehrsthul f\"ur Mathematik VI, Seminargeb\"aude A 5, 6, 68131 Mannheim, Germany.}
\address{Current address: Ruprecht-Karls-Universit\"at Heidelberg, Germany}
\email{abasalaev@mathi.uni-heidelberg.de}

\maketitle

\setcounter{tocdepth}{1}
\tableofcontents

\section{Introduction}
  To a quasi--homogeneous polynomial $W$, having an isolated critical point at the origin, and a group $G$ of diagonal symmetries of $W$, 
  FJRW theory associates the certain moduli space together with a virtual fundamental cycle giving rise to a well--defined intersection theory (see \cite{R}).
  First main application of this moduli space was to the \textit{Witten's equation}. This equation, originating from physics, is due to E.~Witten, but it only became mathematically reasonable on this moduli space of the FJRW theory. The name ``FJRW theory'' stands therefore for H.Fan, T.Jarvis and Y.Ruan, who gave the construction (in \cite{FJR}) and for E.Witten, whose idea was a sparkle for it. 

  This new moduli space can be seen as the generalization on the moduli space of the stable curves. From this point of view FJRW theory can be seen as the cousin of the Gromov--Witten theory.
  It was moreover shown in \cite{FJR} that for $W$ defining ADE singularities, and certain symmetry groups $G$, the partition function of the intersection numbers on this moduli space is a tau--function of the Kac--Wakimoto hierarchy. Then for $W = x^{r+1}$ and cycling group $G$, generated by $g(x) := \exp(2\pi \sqrt{-1}/(r+1)) x$, this new moduli space generalizes the moduli space of the $r$--spin curves, whose Gromov--Witten partition function is a tau--function of the Gelfand--Dykij hierarchy (see \cite{FSZ}).
  
  Another important application of the FJRW theories lies in the area of mirror symmetry. In mirror symmetry the pair $(W,G)$ as above is called \textit{Landau--Ginzburg orbifold}, and FJRW theory provides the A--side model of it. Several mirror symmetry results about the FJRW theories were published in \cite{CR,MR,MS,KS,LLSS,SZ2,PS,BP}.
  Establishing these mirror symmetry results one had to compute certain intersection numbers on the moduli space of the FJRW theory. However, the explicit use of the virtual fundamental cycle appeared to be hard. To our knowledge, in all the examples known, FJRW theory is not computed by using the virtual fundamental cycle of Fan--Jarvis--Ruan itself, but only utilizing the certain properties, it satisfies. These properties were derived already in \cite{FJR}, and called there ``axioms''.
  
  These axioms appeared to be powerful enough for the mirror symmetry purposes, where usually there is no need to compute the theory completely. For all mirror symmetry results above except \cite{BP}, just some small list of intersection numbers was computed on the FJRW theory side. In particular up to now there is no closed formula even for the genus zero potential of any FJRW theory except one particular case in loc. cit.. At the same time even in the computation of the certain intersection numbers, only the most extreme possible symmetry groups $G$ are considered up to now, except one particular case in \cite{SZ2}, --- maximal symmetry groups of $W$.

  The results of this paper come in two groups.
  
  \subsection*{FJRW theory}
  In this paper we take the ``axioms'' of \cite{FJR} as a definition of the FJRW theory. 
  Namely, we consider the FJRW theory as a Cohomological field theory, satisfying certain additional list of axioms.
  We consider the simple--elliptic singularity $\tilde E_7$ represented by $W := x^4 + y^4 + z^2$ with the three symmetry groups:
  \begin{align*}
    G_1 &:= \langle a_1,b_1,c_1 \rangle, \ &&a_1(x,y,z) := \left(\sqrt{-1}x, \sqrt{-1}y, z \right), \ b_1(x,y,z) := (x,-y,z),
    \\
    & && c_1(x,y,z) := (x,y,-z),
    \\
    G_2 &:= \langle a_2,b_2\rangle, \ &&a_2(x,y,z)  := \left(\sqrt{-1}x, \sqrt{-1}y, -z \right), \ b_2(x,y,z):= (x,-y,z),
    \\
    G_3 &:= \langle a_3,b_3\rangle, \ &&a_3(x,y,z) : = \left(\sqrt{-1}x, \sqrt{-1}y, z \right), \ b_3(x.y,z) := (x,y,-z),
  \end{align*}
  All these groups are not maximal for $W$, and this is the first novelty of this paper. All three FJRW theories of $(\tilde E_7,G_k)$ are 6--dimensional.  By using the ``axioms'' of \cite{FJR} only, we reconstruct the genus zero potentials of these FJRW theories up to the scaling of the variables. We give the closed formulae for the three genus zero potentials (see Propositions~\ref{proposition: primary potential of W,G_1}, \ref{proposition: primary potential of W,G_2} and \ref{proposition: primary potential E,G_3}). It turns out that two of these genus zero potentials can be reconstructed from the axioms only up to the scaling. This shows in particular that for the questions, where the particular values of the correlators are important, it's not enough to consider the axioms of FJRW theory only. It turns out also that the third genus zero potential we compute has irrational coefficients. This potential can be written in $\QQ[[\bt]]$ only after a rescaling of the variables.
  
  \subsection*{CY/LG correspondence} 
  Currently, working with the non--maximal symmetry groups on the FJRW theory side makes it hard to speak about the mirror symmetry. This is because the B side should be considered with the non--trivial symmetry group then, and an orbifolded Saito theory is not yet constructed (see \cite{BTW1, BTW2}). 
  However one could anyway consider one mirror symmetry conjecture in this setting too --- the CY/LG correspondence conjecture. It suggests that the partition functions of the two different A--side models, being both mirror dual to the same B--model, are connected by a Givental's action (acting on the space of all partition functions). 
  
  In this paper for the three FJRW theories of the pairs $(\tilde E_7,G_k)$ as above we establish also the CY/LG correspondence. Namely, we provide explicitly the R--matrices of Givental, s.t. up to the certain S--action of Givental the partition function of the FJRW theory is obtained by applying the Givental's action to the partition function of the Gromov--Witten theory of the orbifold $\PP^1_{2,2,2,2} := \left[\E \big/ (\ZZ/2\ZZ\right)]$ --- the orbifold quotient of the elliptic curve by the hyperelliptic involution.
  
  \begin{theorem*}[Theorem~\ref{theorem:main} in the text]
    Up to the certain different Givental's S--actions $S^{(k)}$ the partition functions of the FJRW theories $(\tilde E_7,G_k)$, $k=1,2,3$ are connected to the partition function of the Gromov--Witten theory of $\PP^1_{2,2,2,2}$ by the same Givental's R--action of:
    $$
      R^{\sigma^\prime} := 
	\exp( \left(
	\begin{array}{c c c}
	  0 & \dots & \sigma' \\
	  \vdots & 0 & \vdots \\
	  0 & \dots & 0
	\end{array}
      \right) z), \quad \text{ for } \quad \sigma' = -\frac{1}{2\pi^2}\left(\Gamma(\frac{3}{4})\right)^4,
    $$
    so that holds:
    $$
      {\mathcal Z}^{(\tilde E_7, G_k)} =  \hat R^{\sigma'}  \cdot \hat S^{(k)} \cdot \Z^{\PP^1_{2,2,2,2}}, \quad k=1,2,3.
    $$
  \end{theorem*}
  The S--actions are usually considered to be of little importance because they only stand for the shift of coordinates and a basis choice (in the Chen--Ruan cohomology ring in our case), and hence do not affect ``the geometry'' of the Cohomological field theory. 
  However no explicit computation can be done without knowing these S--actions. Due to this fact we also keep track of them in this paper.

  For the simple--elliptic singularities, CY/LG correspondence conjecture was also considered in \cite{SZ2} in a beautiful manner. It was explained there in terms of a natural operation on the space of quasi--modular forms --- Cayley transform. However \cite{SZ2} didn't derive this particular R--action of Givental giving the CY/LG correspondence or establish the particular Cayley transform. It was first \cite{BP}, where the explicit R--action was given for the simple--elliptic singularities, but with the maximal symmetry group only.
  
  The proof of the theorem uses extensively the explicit formulae for the genus zero potentials of $\PP^1_{4,4,2}$, $\PP^1_{2,2,2,2}$ Gromov--Witten theories and explicitly computed FJRW theories of $(\tilde E_7,G_k)$. We utilize the fact that genus zero potentials of both Gromov--Witten theories can be written via the quasi--modular forms. At the same time, even missing the orbifolded Saito theory, we consider the certain $\SL2C$--action on the space of WDVV equation solutions, that allows us to connect the genus zero partition functions of $\PP^1_{2,2,2,2}$ and $(\tilde E_7,G_k)$. This action was proposed in \cite{BT} as a model for the primitive form change for the Saito theory and was shown to be equivalent to the particular Givental's action in \cite{B}.
  
  \subsection*{Organization of the paper} In Section~\ref{section: FJRW} we define the FJRW theory as a CohFT, subject to the certain list of additional axioms. Gromov--Witten theory of elliptic orbifolds is reviewed in Section~\ref{section: GW theory}. We make certain preparations there, needed to perform the computations. In Section~\ref{section: group actions} we define the group action on the space of CohFTs. Section~\ref{section: CYLG} is devoted to the CY/LG correspondence, where we give the proof of the main theorem with the help of computations, performed in Section~\ref{section: proofs}. This is the last section too, where we give explicit formulae for the primary potentials of the FJRW theories of $(\tilde E_7, G_k)$, $k=1,2,3$ as above --- see Propositions~\ref{proposition: primary potential of W,G_1}, \ref{proposition: primary potential of W,G_2} and \ref{proposition: primary potential E,G_3}. Certain useful formulae are given in Appendix.
  
  \subsection*{Acknowledgement}
  The work of A.B. was partially supported by the DFG grant He2287/4--1 (SISYPH). The author is also grateful to Nathan Priddis, Amanda Francis and Yefeng Shen for the useful discussions and email correspondence.
 
\section{FJRW theory}\label{section: FJRW}

In this section we define the FJRW theory axiomatically as a Cohomological field theory $\Lambda^{(W,G)}$, satisfying some additional system of axioms, as given in Theorem~4.1.8 of \cite{FJR}. In this way all our conclusions hold true for the FJRW theories of $(W,G)$, defined through the virtual fundamental cycle. At the same time it's important to note that to our knowledge almost all computations done up to now in FJRW theories only use these ``axioms'' of \cite{FJR}.

\subsection{The pair $(W,G)$}\label{section: FJRW notations}

  Throughout this paper let $W = W(\bx) = W(x_1,\dots,x_N) \in \CC[\bx]$ be a quasi--homogeneous polynomial. Namely there are integers $d,w_1,\dots,w_N$, s.t. $\gcd(w_1,\dots,w_N) = 1$,
  and for any $\lambda \in \CC^*$ holds $W(\lambda^{w_1}x_1,\dots,\lambda^{w_N}x_N) = \lambda^d W(x_1,\dots,x_N)$. Denote $q_k := w_k/d$ for $k=1,\dots,N$.
  Assume also $0 \in \CC^N$ to be an isolated critical point of $W$ and the weight set to be unique.
  
  Let $G_W := \{\alpha \in (\CC^*)^N \mid W(\alpha \cdot \bx) = W(\bx)\}$ be the so--called {\it maximal group of symmetries} of $W$ (or just $G_{max}$ is the polynomial is clear from the context).
  It's non--empty as $W$ is quasihomogeneous. Denote $e[\alpha] := \exp(2 \pi \sqrt{-1} \alpha)$ for any $\alpha \in \QQ$. 
  Then for ${J:=(e[q_1],\dots,e[q_N])}$, the group $\langle J \rangle$  is a non--empty subgroup of~$G_W$.

  The group $G \subseteq G_W$ is called {\it admissible} if $\langle J \rangle \subseteq G$. 
  In what follows, we will assume $d$, the degree of $W$, to be also the exponent of $G_W$, i.e. for each $h\in G_W$, $h^d=\id$. 
  This is not the case in general, but holds in our examples.

\subsection{Cohomological field theories}\label{subsection:cohft axioms}
Let $(V,\eta)$ be a finite--dimensional vector space with a non--degenerate pairing.
Consider a system of linear maps 
\[
\Lambda_{g,n}: V^{\otimes n} \rightarrow H^*(\M{g}{n}),
\] 
defined for all $g,n$ such that $\M{g}{n}$ exists and is non--empty.
The set $\Lambda_{g,n}$ is called a \emph{cohomological field theory on $(V,\eta)$}, or CohFT, if it satisfies the following axioms.

\begin{cohftaxiom}
$\Lambda_{g,n}$ is equivariant with respect to the $S_n$--action, permuting the factors in the tensor product and the numbering of marked points in $\M{g}{n}$.
\end{cohftaxiom}

\begin{cohftaxiom}
For the gluing morphism $\rho: \M{g_1}{n_1+1} \times \M{g_2}{n_2+1} \rightarrow \M{g_1+g_2}{n_1+n_2}$ we have:
\[
\rho^* \Lambda_{g_1+g_2,n_1+n_2} = (\Lambda_{g_1, n_1+1} \cdot \Lambda_{g_2,n_2+1}, \eta^{-1}),
\]
where we contract with $\eta^{-1}$ the factors of $V$ that correspond to the node in the preimage of~$\rho$.
\end{cohftaxiom}
      
\begin{cohftaxiom}
For the gluing morphism $\sigma: \M{g}{n+2} \rightarrow \M{g+1}{n}$ we have:
\[
\sigma^* \Lambda_{g+1,n} = (\Lambda_{g, n+2}, \eta^{-1}),
\]
where we contract with $\eta^{-1}$ the factors of $V$ that correspond to the node in the preimage of~$\sigma$.
\end{cohftaxiom}
    
In this paper we further assume the CohFT $\Lambda_{g,n}$ to be \emph{unital} --- i.e. there is a fixed vector $\textbf 1 \in V$ called the \emph{unit} such that the following axioms are satisfied.
  
\begin{unitaxiom} For every $\alpha_1,\alpha_2 \in V$ we have: $\eta(\alpha_1, \alpha_2) = \Lambda_{0,3}(\textbf{1} \otimes \alpha_1 \otimes \alpha_2)$.
\end{unitaxiom}

\begin{unitaxiom} Let $\pi: \M{g}{n+1} \rightarrow \M{g}{n}$ be the map forgetting the last marking, then:
\[
\pi^* \Lambda_{g,n}(\alpha_1 \otimes \dots \otimes \alpha_n) = \Lambda_{g,n+1}(\alpha_1 \otimes \dots \otimes \alpha_n \otimes \textbf{1}).
\]
\end{unitaxiom}
A CohFT $\Lambda_{g,n}$ on $(V,\eta)$ is called \emph{quasihomogeneous} if the vector space $V$ is graded by a linear map $\deg: V \to \QQ$ and there is a number $\delta$, such that for any $\alpha_1, \dots, \alpha_n \in V$ holds:
$$
  \left( (g-1)\delta + n \right)\Lambda_{g,n}(\alpha_1,\dots,\alpha_n) = \left( \frac{1}{2}\deg_{coh} + \sum_k \deg (\alpha_k) \right) \Lambda_{g,n}(\alpha_1,\dots,\alpha_n),
$$
where $\deg_{coh}$ is the (real) $H^*(\M{g}{n})$--cohomology class degree.

 Let $\psi_i \in H^2(\M{g}{n})$,  $1 \le i \le n$ be the so--called psi--classes. The genus $g$, $n$--point correlators of the CohFT are the following numbers:
 \[
  \langle \tau_{a_1}(e_{\alpha_1}) \dots \tau_{a_n}(e_{\alpha_n}) \rangle_{g,n}^\Lambda := \int_{\M{g}{n}} \Lambda_{g,n}(e_{\alpha_1} \otimes \dots \otimes e_{\alpha_n}) \psi_1^{a_1} \dots \psi_n^{a_n}.
 \]
 Denote by $\F_g$ the generating function of the genus $g$ correlators, called genus $g$ potential of the CohFT:
 \[
 \F_g := \sum_{\boldsymbol \alpha, \bf a} \frac{\langle \tau_{a_1}(e_{\alpha_1}) \dots \tau_{a_n}(e_{\alpha_n}) \rangle_{g,n}^\Lambda}{\Aut( \{ \boldsymbol \alpha, \bf a \})} \ t^{a_1,\alpha_1} \dots t^{a_n,\alpha_n}.
 \]
 It is useful to assemble the correlators into a generating function called partition function of the CohFT
 $\Z := \exp \left( \sum\nolimits_{g \ge 0} \hbar^{g-1} \mathcal F_g \right)$.
 We will also make use of the so--called \textit{primary} genus $g$ potential that is a function of the finite number of variables $t^\alpha:=t^{0,\alpha}$ defined as follows:
 \[
   F_g := \F_g \mid_{t^\alpha := t^{0,\alpha}, \ t^{\ell,\alpha} = 0, \forall \ell \ge 1}
 \]
 what is also sometimes called a \textit{restriction to the small phase space}.
   
  Due to some topological properties of $\M{0}{n}$, the small phase space potential of a CohFT on $(V,\eta)$ satisfies the so--called WDVV equation. For any four fixed $1 \le i,j,k,l \le \dim(V)$ holds:
  \begin{equation}\label{eq: wdvv}
    \sum_{p,q = 1}^{\dim(V)} \frac{\partial^3 F_0}{\partial t^i \partial t^j \partial ^p} \eta^{p,q} \frac{\partial^3 F_0}{\partial t^q \partial t^k \partial^l} = \sum_{p,q = 1}^{\dim(V)} \frac{\partial^3 F_0}{\partial t^i \partial t^k \partial ^p} \eta^{p,q} \frac{\partial^3 F_0}{\partial t^q \partial t^j \partial^l}.
  \end{equation}
  It's important to note that function $\F_0$ is reconstructed unambiguously from $F_0$ due to the \textit{topological recursion relation in genus zero}. Hence function $F_0$ contains all genus zero information of the CohFT.
  
  \subsection{Moduli of W--curves}

  An \emph{$n$--pointed orbifold curve} $\C$ is a 1--dimensional Deligne--Mumford stack with at worst nodal singularities with orbifold structure only at the marked points and the nodes. Moreover the orbifold structure is required to be \emph{balanced} at the nodes.

  A $d$--stable curve is a proper connected orbifold curve $\C$ of genus $g$ with $n$ distinct smooth markings $p_1,\dots,p_n$ such that the  $n$--pointed underlying coarse curve is stable, and all the stabilizers at nodes and markings have order $d$. 
  The moduli stack $\overline{\calM}_{g,n,d}$ parameterizing such curves is proper, smooth and has dimension $3g-3+n$. It differs from the moduli space of curves only because of the stabilizers over the normal crossings.

  Let $W$ be written as 
  $$
    \displaystyle W = \sum_{i=1}^M c_i \prod_{k=1}^N x_k^{a_{ik}}, \quad a_{ik}\in \NN, \ c_i\in \CC.
  $$
  Given line bundles $\L_1, \ldots , \L_N$ on the $d$--stable curve $\C$, we define the line bundle 
  \[
    \W_i(\L_1,\dots,\L_N) := \bigotimes_{k=1}^N\L_k^{\otimes a_{ik}}, \quad 1 \le i \le M.
  \] 

  \begin{definition}A \emph{$W$--structure} is the data 
  $
  (\C, p_1,\dots, p_n, \L_1,\dots,\L_N,\varphi_1,\dots \varphi_N),
  $
  where $\C$ is an $n$--pointed $d$--stable curve, the $\L_k$ are line bundles on $\C$ satisfying 
  \[
  \W_i(\L_1,\dots,\L_N)\cong \omega_{\log} = \omega(p_1 + \dots + p_n),
  \] 
  and for each $k$, $\varphi_k:\L_k^{\otimes d}\to \omega_{\log}^{w_k}$ is an isomorphism of line bundles. 
  \end{definition}
  When $G = G_W$, the following theorem holds.
  \begin{theorem}[Fan--Jarvis--Ruan, \cite{FJR}]
    There exists a moduli stack of all $W$--structures, denoted by $\W_{g,n,G_W}(W)$, possessing also the suitable virtual fundamental cycle $[\W_{g,n,G_W}(W)]^{vir}$, defining the CohFT of the pair $(W,G_W)$ by the morphism $\mathrm{st}: \W_{g,n,G_W}(W) \to \M{g}{n}$, forgetting the $W$--structure of a curve.
  \end{theorem}
   
  For the cases when $G \subsetneq G_W$, consider the following construction. Let $Z$ be a Laurent polynomial, satisfying the following three conditions: (i) it's quasi--homogeneous with the same weights $q_k$ as $W$ (see Section~\ref{section: FJRW notations} for the notation), (ii) it has no monomials in common with $W$, (iii) $G = G_{W+Z}$. 
  
  Then one sets: $\W_{g,n,G}(W) := \W_{g,n,G_{W+Z}}(W+Z)$. It turns out that the moduli space obtained is independent of the choice of $Z$.
 
  Moreover, there is a universal curve $\C$ with the projection $\pi: \C \to \W_{g,n,G}$, endowed with the universal $W$--structure $(\LL_1,\dots,\LL_N)$.

  \begin{example}
    For $W = x_1^{r+1}$ and $G = G_W$ we have $\W_{g,n,G_W} \cong \overline{\mathcal{M}}_{g,n}^{r}$ --- the module space of $r$--spin curves.
  \end{example}

  \subsection{FJRW CohFT of a simple--elliptic singularity}
  Denote $\Omega_W := \Omega^N_{\CC^N,0} \Big/ \left(dW \wedge d^{N-1}_{\CC^N,0}\right)$. It's a finite dimensional rank one module over the Jacobian algebra of $W$ in case when $W$ has only isolated critical points. It's equipped with the non--degenerate bilinear form $\langle \cdot, \cdot \rangle_W$ --- the Poincar\'e residue pairing.
  
  For any $h \in G$ denote by $\Fix(h) \subseteq \CC^N$ the fixed locus of $h$ and $N_h := \dim(\Fix(h))$. Define
  $W^h := W\mid_{\Fix(h)}: \CC^{N_h} \to \CC$. We call $h \in G$ s.t. $N_h = 0$ the {\it narrow sector} group elements.
  
  For $N_h \neq 0$ we can consider the module $\Omega_{W^h}$. Because $W^h$ will have only isolated critical points too, $\Omega_{W^h}$ will be finite--dimensional, equipped with the non--degenerate bilinear form $\langle \cdot, \cdot \rangle_{W^h}$. It also has a (coordinate--wise) $G$--action on it. Denote $\Omega_h := \left( \Omega_{W^h} \right)^G$ --- the $G$--invariant subspace of $\Omega_{W^h}$. 
  
  If $N_h = 0$ we set $\Omega_h := \CC \cdot e_1$ with the trivial $G$--action, s.t. $(\Omega_h)^G = \Omega_h$. It's also assumed to have the bilinear form on it. Namely, $\langle e_1, e_1\rangle_{W^h} := 1$.
  
  Note that $\Fix(h) = \Fix(h^{-1})$. Let $\psi_h$ be an isomorphism $\Omega_h \cong \Omega_{h^{-1}}$.

  \begin{definition}
    We call a unital CohFT $\Lambda = \Lambda_{g,n}^{(W,G)}$ a FJRW CohFT of $(W,G)$ if it satisfies the following list of axioms \ref{subsection: FJRW state space} --- \ref{subsection: FJRW concavity}.
  \end{definition}

  \subsubsection{State space}\label{subsection: FJRW state space}
  $\Lambda$ is a CohFT on the state space $\H_{W,G} := \oplus_{h \in G} \H_h$, where as a vector space $\H_h \cong \Omega_h$ for all $h \in G$. Equip $\H_{W,G}$ with the $\CC$--bilinear pairing $\langle \cdot,\cdot\rangle_{W,G} := \oplus_{h\in G}\langle \cdot, \cdot \rangle_h$, for $ \langle \cdot, \cdot \rangle_h: \ \H_h \otimes_{\CC} \H_{h^{-1}} \to \CC$ defined by $\langle \cdot,\cdot \rangle_h:= \langle \cdot, \psi_{h^{-1}}(\cdot) \rangle_{W^h}$. This pairing is non--degenerate too.

  In what follows for any $h \in G$ by the element $\alpha_h \in \H_{W,G}$ we will always assume a vector, belonging to $\H_h \subset \H_{W,G}$.
  
  For any $h \in G$, let the numbers $\varTheta_k^h \in \QQ \cap [0,1)$ be s.t. $h$ is represented by the diagonal $\mathrm{GL}(N,\CC)$--matrix $\mathrm{diag}(e[\varTheta_1^h],\dots,e[\varTheta_N^h])$.
  
  The vector space $\H_{W,G}$ is graded by $\deg_W: \H_{W,G} \to \QQ$, defined by
  $$
    \deg_W(\alpha_h) := N_h + 2 \ \iota(g), \quad \alpha_h \in \H_h,
  $$
  where the \textit{degree shifting number} $\iota(h)$ is defined as follows.
  $$
    \iota(h):=\sum_{k=1}^N(\varTheta_k^h-q_k).
  $$
  
  \subsubsection{Degree}\label{subsection: FJR degree axiom}
  Set $\hat c := \sum_{k=1}^N (1 - 2q_k) \in \QQ$. The class $\Lambda_{g,n}^{(W,G)}(\alpha_{h_1},\dots, \alpha_{h_n})$ vanishes
  unless $\hat c(g-1) + \sum_i \iota_{h_i} \not\in \ZZ$.
  Otherwise it has the following degree
  \[
    2\left((\hat c -3)(1-g)+ n -\sum_{i=1}^n \iota(h_i) - \sum_{i=1}^n \frac{N_{h_i}}{2}\right).
  \]

  \subsubsection{Selection rule}\label{subsection: FJR selection rule}
  The class $\Lambda_{g,n}^{(W,G)}(\alpha_{h_1},\dots,\alpha_{h_n})$ is zero unless for all $1 \le k \le N$ holds:
  \[
    q_k(2g-2+n)-\sum_{i=1}^n \varTheta_k^{h_i} \in \ZZ
  \]
  
  \subsubsection{$G_W$--invariance}\label{subsection: FJR Gw invariance}
  Assume axiom~\ref{subsection: FJRW state space} to hold true.
  Consider the action of $G_W$ on each $\Omega_h$, and extend it to the action of $G_W$ on $\H_{W,G}$. The CohFT $\Lambda^{(W,G)}_{g,n}$ (considered as a system of linear maps) is required to be invariant under this action.

  \subsubsection{Concavity}\label{subsection: FJRW concavity}
  Suppose that $h_i\in G$ are s.t. $\Fix(h_i) = \emptyset$ for all $i = 1,\dots,n$.
  Let $\pi$ be the projection from the universal curve of the moduli space and $\LL_1,\dots,\LL_N$ be the universal $W$--structure. Let $c_{top}$ stand for the top Chern class.
  If $\pi_*\left(\bigoplus_{k=1}^N \LL_k \right)=0$, then holds:
  \[
    \Lambda_{g,n}^{(W,G)}(\alpha_{h_1},\dots,\alpha_{h_n}) = \frac{|G|^g}{\deg(\mathrm{st})} \mathrm{PD} \ \mathrm{st}_* c_{top}\left(\big(R^1\pi_*\bigoplus_{k=1}^N \LL_k \big)^\vee\right).
  \]
  The subspace of $\H_{W,G}$, generated by $\alpha_{h_1},\dots,\alpha_{h_n}$ is called \textit{concave}.

  \subsection{Remarks on the axioms}
  
  The state space axiom is usually introduced via the so--called {\it Lefschetz thimbles} of $W^h$. However they are only used further as the generators of the vector spaces, that are isomorphic to those we used --- $\Omega_h$.
  
  Degree axiom we formulate, is exactly Degree axiom of Fan--Jarvis--Ruan, modulo the notational difference. We give only the degrees of the cohomology classes in $\M{g}{n}$ while in \cite{FJR} the state space degrees (that of Lefschetz thimbles, treated as homology classes) are counted too.
  
  It's immediate to note that the CohFT $\Lambda^{(W,G)}$ is quasi--homogeneous with $\delta := 3- \hat c$ and the grading $\deg_W$ on $\H_{W,G}$. It's also unital with the unit --- the generator of $\Omega_{J}$ (which is one--dimensional because $\Fix(J) = \emptyset$).

  \subsection{Concavity axiom}
  In the list of axioms above it's clear that the only source of non--zero quantitative data of FJRW CohFT is concavity axiom and pairing axiom. The latter one only concerns the three point correlators $\langle \ \rangle_{0,3}$, hence this is only concavity axiom giving us the ``data''. It's a surprising fact, that this is indeed the concavity axiom, providing all non--trivial computations of all mirror symmetry results, we reference in this paper. In other words, this small source of data appeared to be powerful enough for the mirror symmetry needs.
  
  After the result of A.Chiodo (\cite[Theorem~1.1.1]{C}) the $\M{g}{n}$--cohomology class of concavity axiom can be written via the well--known $\M{g}{n}$ tautological classes --- $\kappa_d$, $\psi_k$, classes of the divisors. In particular for $W = x_1^4 + x_2^4 + x_3^2$ and $G = G_W$ we have:
  \[
   \Lambda_{0,4}^{(W,G_W)}(\alpha_{h_1}, \alpha_{h_2},\alpha_{h_3},\alpha_{h_4}) = \frac{1}{2} \sum_{i=1}^3 \left(B_2(q_i)\kappa_1 - \sum_{j=1}^3 B_2(\theta_i^{h_j})\psi_j  + \sum_{\Gamma} B_2(\theta_i^{h_\Gamma}) [\Gamma] \right),
  \]
  where $B_2(z) := z^2 - z + 1/6$, $[\Gamma]$ is a class of the divisor in $\M{0}{4}$ and the summation is taken over the possible decorations of such a divisor.
  Consult \cite[Section~3]{F} for details.
  
  \subsection{FJRW theory of a simple--elliptic singularity}
%
  Fixing the basis $\{ \phi_k^{(h)}(\bx)d^{N_h}\bx \}$ of $\Omega_h $ for all $h \in G$, we will consider the basis $\left\lbrace [h,\phi_k^{(h)}(\bx)] \right\rbrace_{h,k}$ of $\H_{W,G}$. For narrow $h \in G$, s.t. $N_h = 0$ we denote $\alpha_h \in \H_h \subset \H_{W,G}$ by $[h,1]$.
  
  Associate also to the vector $[h,\phi_k^{(h)}(\bx)]$ the variable $t_{\phi_k^{(h)}(\bx),h}$ if $N_h \neq 0$ and the variable~$t_{h}$ to $[h,1]$.

  In the case of simple--elliptic singularities concavity axiom is in particular powerful.
  
  \begin{proposition}
    Let $W = x_1^4 + x_2^4 + x_3^2$ define a simple--elliptic singularity and $G$ be any admissible group of its symmetries. Then for any $h_1,\dots,h_n$, s.t. $N_{h_k} = 0$ for all $1 \le k \le n$ the subspace generated by $\alpha_{h_1},\dots,\alpha_{h_n}$ is concave.
  \end{proposition}
  \begin{proof}
    The proof copies proof of Proposition~1.6 in \cite{PS}. It's enough to count the line bundle degrees of $\L_k$. Because $\sum_{k=1}^3 q_k = 1$ and $q_k < 1$ for a point $(\C,p_1,\dots, p_n,\L_1,\L_2,\L_3, \phi_1,\phi_2,\phi_3)$ on each irreducible component $\C_v$ of $\C$ holds
    $$
      \deg(|\L_k|_{\C_v}) \le q_k \left(\text{\# nodes}(\C_v) -2\right) < \text{\# nodes}(\C_v) -1,
    $$
    where $|\L_k|$ denotes the pushforward of $\L_k$ to the underlying curve of $\C$. The inequality obtained finally shows that $|\L_k|$ has no section.
  \end{proof}
  \begin{corollary}\label{corollary: concave sectors}
    For a simple--elliptic singularity $W$ let $F_0^{(W,G)}$ and $F_0^{(W,G_W)}$ be the genus zero primary FJRW potentials of $(W,G)$ and $(W,G_W)$ respectively. Then holds:
    $$
      F_0^{(W,G)} \mid_{t_{\phi,h} = 0, \ h \not\in G^{nar}} \quad = \quad F_0^{(W,G_W)} \mid_{t_{\phi,h} = 0, \ h \not\in G^{nar}}
    $$
  \end{corollary}
  \begin{proof}
    The full state space $\H_{W,G_W}$ is concave.
    As the vector space $\H_{W,G}$ is defined as the direct sum over all $G$ elements, if $\alpha_h \in \H_{W,G}$, then there is a vector $\alpha_h' \in \H_{h} \subset \H_{W,G_W}$. 
    These two vectors can be identified because $\Omega_h \cong \CC$.
    The rest follows from Concavity axiom because the formula for the correlators of $\Lambda_{0,n}^{(W.G)}$ and $\Lambda_{0,n}^{(W,G_W)}$ is literally the same.
  \end{proof}

\section{Gromov--Witten theory of elliptic orbifolds}\label{section: GW theory}
  In this paper we make use of the \textit{orbifold Gromov--Witten} (that we call later just GW theory). Like FJRW theory, GW theory also defines certain CohFT. The state space of it is the orbifold cohomology ring, or Chen--Ruan cohomology ring, and the CohFT is fixed by the (Poincare dual to the pushforward of) virtual fundamental class of the moduli space of stable maps. 
  
  We skip completely the definition of the Gromov--Witten theory here, referencing an interested reader to \cite{A}. For the cases we only need in this paper --- of the elliptic orbifolds, we define the Gromov--Witten theory in genus zero by giving explicitly the CohFT potentials, found in \cite{ST,BP,SZ2}.
  
  The so--called \textit{elliptic orbifolds} $\PP^1_{a,b,c}$ for $(a,b,c) = (3,3,3)$, $(4,4,2)$ or $(6,3,2)$ --- are smooth orbifold projective lines with only $3$ points having the non--trivial orbifold structure $\ZZ/a\ZZ$, $\ZZ/b\ZZ$ and $\ZZ/c\ZZ$. They are called elliptic because each of them can be realized as a global quotient of the elliptic curve by the finite group action. GW theory of these orbifolds was found to give the A--model, mirror to the Saito structures B--model of the simple--elliptic singularities (see the references, given in Introduction).
  
  Apart from the three elliptic orbifold named, there is one more, more mysterious one -- $\PP^1_{2,2,2,2}$. This orbifold is obtained as a global quotient of an elliptic curve by the hyperelliptic involution. Compared to the previously named elliptic orbifolds, this one was not identified in the context of mirror symmetry until the recent result of \cite{SZ2}.
  
  In what follows denote $\X_2 := \PP^1_{2,2,2,2}$ and $\X_4 := \PP^1_{4,4,2}$. Fix the bases of the Chen--Ruan cohomology $H^*_{orb}(\X_k)$ as follows.
  
  Let $\Delta_0, \Delta_{-1}$ be the degree $0$ and degree $2$ generators of $H^*(\PP^1)$ respectively, viewed as untwisted sector of $H^*_{orb}(\X_k)$. Let $\Delta_{i,j}$ be the twisted sector generators, corresponding to the $i$--th point with a non--trivial isotropy group. We have:
  $$
    H^*_{orb}(\X_2) \cong \QQ\Delta_0\oplus\QQ\Delta_{-1} \ \bigoplus_{i=1}^4 \QQ\Delta_{i,1},
    \
    H^*_{orb}(\X_4) \cong \QQ\Delta_0\oplus\QQ\Delta_{-1} \ \bigoplus_{j=1}^3 \QQ\Delta_{1,j} \bigoplus_{j=1}^3 \QQ\Delta_{2,j} \bigoplus \QQ \Delta_{3,1}.
  $$
  The ring $H^*_{orb}(\X_k)$ is also endowed with the pairing $\eta$, an analogue of the Poincar\'e pairing. Gromov--Witten theory of $\X_k$ expresses the intersection theory of the moduli space of the stable orbifold maps to $\X_k$. 
  We will be only working with the CohFT it defines on the moduli space of stable curves. 
  
  The genus $0$ potential of the Gromov--Witten theory of $\X_k$ is a function of the variables $\bt$, being dual to the basis element fixed, and also of the formal {\it Novikov variable} $q_{formal}$. We will fix the variables $\bt$ differently in what follows, but we always keep $t_0,t_{-1}$ to correspond to the basis elements $\Delta_0, \Delta_{-1}$ respectively.
  
  \subsection{Novikov variable}
  The Novikov variable $q = q_{formal}$ is used to keep track of the homology class --- it appears in the genus $g$ potential as $q^\beta$, where $\beta \in H_2(X)$. In our case $\dim(H_2(\X_k)) = 1$ and by using Divisor equation (of the GW theory) the Novikov variable $q$ can be identified with $\exp(t_{-1})$ (cf. \cite[Section 1.2]{SZ1}). The correlation functions of the genus $0$ potentials  after such an identification appear to coincide with the Fourier expansions of the certain functions. However it's useful to work with the function itself rather than the Fourier expansion of it.
  To do this we make another identification of the Novikov variable that depends on the orbifold in question:
  \begin{equation}\label{eq: Novikov to modular}
    q_{formal} = \exp(t_{-1}) = \exp \left( \frac{2 \pi \sqrt{-1} \tau}{k} \right) =: q_k, 
    \quad
    \text{for the orbifold } \X_k.
  \end{equation}
  This identification also affects the cubic terms of the partition function, fixed by the pairing in Axiom~U1. Because of this we can't just take the change of the variables $t_{-1} = 2\pi \sqrt{-1} \tau /k$ (what would change the CohFT state space) and will treat this identification carefully.

  At the same time only after making an identification of the formal variable we get the clear holomorphicity property of the genus zero potential and are able to introduce suitable group action, we use later in the text. For this purpose we introduce new functions --- \textit{analytic potentials} of $\PP^1_{2,2,2,2}$ and $\PP^1_{4,4,2}$ GW theories in order to make the statements about the genuine genus zero potentials. One can do the same for the remaining two elliptic orbifolds $\PP^1_{3,3,3}$ and $\PP^1_{6,3,2}$ as well.
  
  \subsection{Gromov--Witten theory of $\PP^1_{2,2,2,2}$}
  The genus zero potential of this GW theory was found explicitly by Satake--Takahashi in \cite{ST}. We present their result here in a slightly modified form that will be useful for us in what follows.
  
  Let the variables $\{ t_0,t_{-1}, t_1,t_2,t_3,t_4 \}$ be dual to the following basis of $H^*_{orb}(\PP^1_{2,2,2,2})$ (recall the notation above)
  $$
    \left\lbrace\Delta_0, \Delta_{-1}, \frac{1}{\sqrt{2}}\left(\Delta_{2,1}-\Delta_{4,1} \right), \frac{1}{\sqrt{2}}\left(\Delta_{2,1}+\Delta_{4,1} \right), \frac{1}{\sqrt{2}}\left(\Delta_{1,1}-\Delta_{3,1} \right), \frac{1}{\sqrt{2}}\left(\Delta_{1,1}+\Delta_{3,1} \right)\right\rbrace.
  $$
  Consider the functions $\psi_k$, defined by the following formal series in $q$:
  \begin{align*}
    \psi _2(q) := \frac{1}{2} + 2\sum_{n=1}^\infty (-1)^{n-1}\frac{2n q^n}{1-q^n} ,
    \
    \psi _3(q) := 2\sum_{n=1}^\infty (-1)^{n-1}\frac{2n q^{n/2}}{1-q^n}, 
    \ 
    \psi _4(q) := -2\sum_{n=1}^\infty \frac{2n q^{n/2}}{1-q^n}.
  \end{align*}
  In the basis fixed the primary genus zero potential of the GW theory in question assumes the following form:
  \begin{align*}
    F_0^{\PP^1_{2,2,2,2}} &= \frac{1}{2} t_0^2 t_{-1} +\frac{1}{4} t_0 \sum_{k=2}^5 t_k^2  
    -\frac{1}{16} \left(t_3^2 t_4^2 + t_1^2 t_2^2 \right) \psi_4 \left(q^2 \right) -\frac{1}{16} \left(t_1^2 t_3^2 + t_2^2 t_4^2 \right) \psi_2 \left(q^2\right)
    \\
    &  - \frac{1}{16} \left( t_2^2 t_3^2 + t_1^2 t_4^2 \right) \psi_3 \left(q^2\right) - \frac{1}{96} \left( \sum_{k=2}^5 t_k^4 \right) \left( \sum_{k=2}^4 \psi_k \left(q^2\right) \right), \quad q = \exp(t_{-1}).
  \end{align*}
  WDVV equation on this genus zero potential is equivalent to the following system of PDE's on the functions $\{X_2(q),X_3(q),X_4(q)\}$, satisfied by the triple $\{ \psi_2(q^2), \psi_3(q^2), \psi_4(q^2) \}$:
  \begin{equation}\label{eq: Halp in q}
    \begin{aligned}
      q\frac{\p }{\p q} X_2(q) &= X_2(q) \left(X_3(q) + X_4(q) \right) - X_3(q) X_4(q),
      \\
      q\frac{\p }{\p q} X_3(q) &= X_3(q) \left(X_2(q) + X_4(q) \right) - X_2(q) X_4(q),
      \\
      q\frac{\p }{\p q} X_4(q) &= X_4(q) \left(X_2(q) + X_3(q) \right) - X_2(q) X_3(q),
    \end{aligned}
  \end{equation}
  that we call a Halphen's system of equations.

  Note that up to now we didn't use the relation between $q$ and $t_{-1}$. 
  For all $\tau \in \HH$ let the Jacobi theta constants $\vartheta_k(\tau)$ be the holomorphic functions on $\HH$ given by the following Fourier series:
  \begin{equation*}
    \begin{aligned}
      \vartheta_2(\tau) := \sum_{n = -\infty}^\infty e^{\pi \sqrt{-1} \tau (n-1/2)^2},
      \quad
      \vartheta_3(\tau) := \sum_{n = -\infty}^\infty e^{\pi \sqrt{-1} \tau n^2} ,
      \quad
      \vartheta_4(\tau) := \sum_{n = -\infty}^\infty (-1)^n e^{\pi \sqrt{-1} \tau n^2} .
    \end{aligned}
  \end{equation*}
  The function $\vartheta_1(\tau)$ is skipped because it vanishes identically. 
  Consider the functions:
  $$
    X_k^\infty(\tau) := 2 \p_\tau \log \vartheta_k(\tau), \quad X_k^\infty \left( q \right) := \frac{1}{\pi \sqrt{-1}} X_k^\infty \left( \frac{\tau}{\pi \sqrt{-1}} \right), \quad k = 2,3,4.
  $$
  Then the triple $\{X_2^\infty(\tau),X_3^\infty(\tau),X_4^\infty(\tau)\}$ is a solution of Haplhen's system of equations:
  \begin{equation}\label{eq: Halp in tau}
  \begin{aligned}
    \frac{\p }{\p \tau} X_2(\tau) = X_2(\tau) \left(X_3(\tau) + X_4(\tau) \right) - X_3(\tau) X_4(\tau).
    \\
    \frac{\p }{\p \tau} X_3(\tau) = X_3(\tau) \left(X_2(\tau) + X_4(\tau) \right) - X_2(\tau) X_4(\tau),
    \\
    \frac{\p }{\p \tau} X_4(\tau) = X_4(\tau) \left(X_2(\tau) + X_3(\tau) \right) - X_2(\tau) X_3(\tau),
  \end{aligned}
  \end{equation}
  and $\{ X_2^\infty(q), X_3^\infty(q), X_4^\infty(q) \}$ give solution to Eq.~\eqref{eq: Halp in q}.
  We have the equality:
  $$
    \pi \sqrt{-1} \psi_k(q) = X_k^\infty(\tau).
  $$
  
  \begin{notation}\label{notation: analytic potential X2}
    In what follows we denote by $F_{an}^{\PP^1_{2,2,2,2}}$ the analytic potential of $\PP^1_{2,2,2,2}$:
    \begin{align*}
      F_{an}^{\PP^1_{2,2,2,2}} &= \frac{1}{2} t_0^2 \tau +\frac{1}{4} t_0 \sum_{k=2}^5 t_k^2  
      -\frac{1}{16} \left(t_3^2 t_4^2 + t_1^2 t_2^2 \right) X_4^\infty(\tau) -\frac{1}{16} \left(t_1^2 t_3^2 + t_2^2 t_4^2 \right) X_2^\infty(\tau)
      \\
      &  - \frac{1}{16} \left( t_2^2 t_3^2 + t_1^2 t_4^2 \right) X_3^\infty(\tau) - \frac{1}{96} \left( \sum_{k=2}^5 t_k^4 \right) \left( \sum_{k=2}^4 X_k^\infty(\tau) \right).
    \end{align*}
  \end{notation}
  
  \begin{proposition}
    The function $F_{an}^{\PP^1_{2,2,2,2}}$ is holomorphic on $\CC^5\times\HH$ and is solution to the WDVV equation.
  \end{proposition}
  \begin{proof}
    This is straightforward by using the definition of the function $X_k^\infty(\tau)$, Eq.~\eqref{eq: Halp in q} and the properties of $F_0^{\PP^1_{2,2,2,2}}$.
  \end{proof}

  The connection between the functions $F_0^{\PP^1_{2,2,2,2}}$ and $F_{an}^{\PP^1_{2,2,2,2}}$ is obvious --- we have applied the relation $q_{formal} = q_k(\tau)$, however in order to obtain the function, that is solution to the WDVV equation, we had to make an additional rescaling. In what follows we are going to use the second function (having only an indirect connection to the GW theory) in order to make statement about the first function (being indeed a true potential of the GW theory).
  
  Comparing to the functions $\psi_k(q)$ and $X_k^\infty(q)$, big advantage of the functions $X_k^\infty(\tau)$ is that they are holomorphic in $\HH$. Apart from the holomorphicity property, the functions $X_k^\infty(\tau)$ enjoy another major advantage --- there is a $\SL2C$ group action on the space of solutions to the Halphen's system Eq.\eqref{eq: Halp in tau} written in $\tau$, but not on that of Eq.~\eqref{eq: Halp in q}.

  \subsection{Gromov--Witten theory of $\PP^1_{4,4,2}$}
  We write this GW theory in the basis $\Delta_{i,j}$, we have considered at the start of the section. Let also the coordinates $t_{i,j}$ be corresponding to this basis elements. The genus $0$ potential of this orbifold is written completely via the functions $x(q)$, $y(q)$, $z(q)$ and $w(q)$, defined by:
  \begin{align*}
    \frac{1}{4} x(q) := \langle \Delta_{1,1}, \Delta_{1,1}, \Delta_{1,2} \rangle_{0,3}, \quad \frac{1}{4} y(q) := \langle \Delta_{1,2}, \Delta_{2,1}, \Delta_{2,1} \rangle_{0,3}
    \\
    - \frac{1}{8} w(q) := \langle \Delta_{1,1},\Delta_{1,1},\Delta_{1,3},\Delta_{1,3}\rangle_{0,4}, \quad \frac{1}{4} z(q) := \langle \Delta_{1,1}, \Delta_{2,1}, \Delta_{3,1} \rangle_{0,3}.
  \end{align*}
  The functions $x(q)$, $y(q)$, $z(q)$, $w(q)$ have the following expression:
  \begin{align*}
    x(q) &= \left( \vartheta_3(q^8) \right)^2, \ y(q) = \left( \vartheta_2(q^8) \right)^2, \ z(q) = \left( \vartheta_2(q^4) \right)^2, 
    \\ 
    w(q) &= \frac{1}{3} \left( f(q^4) - 2 f(q^8) + 4 f(q^{16}) \right)
  \end{align*}
  for the functions $\vartheta_k(q)$ as above and $f(q) := 1 - 24 \sum_{k=1}^\infty \dfrac{k q^k}{1-q^k}$.
  \begin{proposition}[Appendix~A in \cite{BP} and Section~3.2.3 in \cite{SZ1}]\label{prop: gw potential basics}
    The potential $F^{\PP^1_{4,4,2}}_0$ has an explicit form via the functions defined above. Namely there exists the polynomial $\mathrm{P}_{poly}^{\PP^1_{4,4,2}} = \mathrm{P}_{poly}^{\PP^1_{4,4,2}}(t_0,t_{-1},t_{i,j},x,y,z,w) \in \QQ \left[t_0,t_{-1},t_{i,j}, x,y,z,w \right]$, s.t.
    \[
      F^{\PP^1_{4,4,2}}_0 (t_0,t_{-1},t_{i,j},q) = \mathrm{P}^{\PP^1_{4,4,2}}_{poly} (t_0,t_{-1}, t_{i,j}, x(q), y(q), z(q), w(q)),
    \]
    for $x(q)$, $y(q)$, $z(q)$ and $w(q)$ as above. 
    Moreover the following homogeneity property holds:
     \[
       \mathrm{P}^{\PP^1_{4,4,2}}_{poly} \left( t_0,t_{-1},t_{i,j}, x,y,z,w \right) = \frac{1}{\alpha^{2}} \mathrm{P}^{\PP^1_{4,4,2}}_{poly} \left( t_0, \alpha^2 \cdot t_{-1}, \alpha \cdot t_{i,j}, \frac{x}{\alpha},\frac{y}{\alpha},\frac{z}{\alpha},\frac{w}{\alpha^2} \right),
     \]
     for any $\alpha \in \CC^*$.
  \end{proposition}
  To make the exposition complete, we give also the potential $F_0^{\PP^1_{4,4,2}}$ in Appendix~\ref{section: appendix GW}.

  In what follows the function $z(q)$ will be sometimes skipped because the following identity holds:
  $$
    z(q)^2 = 4 x(q) y(q).
  $$
  It was found by Shen--Zhou \cite{SZ1} that
  WDVV equation on this genus $0$ potential is equivalent to the following system (written in the Novikov variable)
  \begin{equation}\label{eq: wdvv x_4 in q}
  \begin{aligned}
    q \frac{\p}{\p q} x(q) &= 2 x(q) y(q)^2 -x(q) (x(q)^2 -w(q) ),
    \\
    q \frac{\p}{\p q} y(q) &= 2 x(q)^2 y(q)-y(q) (x(q)^2 -w(q)),
    \\
    q \frac{\p}{\p q} w(q) &= w(q)^2-x(q)^4.
  \end{aligned}
  \end{equation}
  The functions $\vartheta_k(q)$ and $\psi_k(q)$ are connected by the certain equalities (see Appendix~\ref{appendix A}).
  Using also double argument formulae for $\vartheta_k$ and comparing the formal series expansions we find:
  \begin{equation}\label{eq: xyz via Halphen}
    \begin{aligned}
      x(q) &= \frac{1}{2} \left(\sqrt{2\psi _2 (q^4)-2\psi _4(q^4)}+ \sqrt{2\psi _2(q^4)-2\psi _3(q^4)}\right),
      \\
      y(q) &= \frac{1}{2} \left( \sqrt{2\psi _2(q^4)-2\psi _4(q^4)}- \sqrt{2\psi _2(q^4)-2\psi _3(q^4)}\right),
      \\
      w(q) &= \psi _2(q^4)+ \frac{1}{2}\psi _3(q^4)+ \frac{1}{2}\psi _4(q^4)+ \sqrt{(\psi _2(q^4)-\psi _3(q^4))(\psi _2(q^4)-\psi _4(q^4))}.
    \end{aligned}
  \end{equation}
  The square roots in the equation above can be unambiguously resolved as being applied to the formal power series in $q$ with the $\QQ_+$ coefficients.
  \begin{proposition}
    WDVV equation on the genus $0$ GW potential of $\PP^1_{4,4,2}$ is equivalent to the Halphen's system of equations.
  \end{proposition}
  \begin{proof}
    This is an easy computation by using Eq.~\eqref{eq: wdvv x_4 in q} and Eq.~\eqref{eq: xyz via Halphen}.
  \end{proof}

  It was found in \cite{SZ1}, that the WDVV equation for the other elliptic orbifolds, $\PP^1_{3,3,3}$ and $\PP^1_{6,3,2}$, can be written in the form similar to Eq.~\eqref{eq: wdvv x_4 in q}. So, there is a special system of ODE's in $q$ for each elliptic orbifold, that is equivalent to the WDVV equation.
  This is not a subject of this paper, however there is a strong evidence to conjecture that WDVV equation for the genus zero potentials of GW theory of all elliptic orbifolds (namely, for $\PP^1_{3,3,3}$ and $\PP^1_{6,3,2}$  too) is also equivalent to Halphen's system of equations. Namely, we believe, that there is a proposition like the one above for the other two elliptic orbilds too.
    
  \begin{notation}\label{notation: X_4}
    Fixing some branch of the square root, denote $\lambda_2:= \sqrt{\pi \sqrt{-1}}$ and $\lambda_4:= \lambda_2/\sqrt{2}$. We have then $\lambda_2^2 = 2\pi \sqrt{-1}/$ and $\lambda_4^2 = 2 \pi \sqrt{-1} /4$.
    For $q(\tau) = \exp \left(\frac{2 \pi \sqrt{-1}}{4}\tau \right)$ introduce the functions:
  \begin{align*}
    x^\infty(\tau) = \lambda_4 \cdot x(q(\tau)), &\quad y^\infty(\tau) = \lambda_4 \cdot y(q(\tau)),   
    \quad
    z^\infty(\tau) = \lambda_4 \cdot z(q(\tau)), 
    \\
    & w^\infty(\tau) = \lambda_4^2 \cdot w(q(\tau)).
  \end{align*}      
    Recall Proposition~\ref{prop: gw potential basics}. We call the function $F_{an}^{\PP^1_{4,4,2}}$ the analytic potential of $\PP^1_{4,4,2}$:
    $$
      F_{an}^{\PP^1_{4,4,2}}(t_0,\tau,t_{i,j}) := \mathrm{P}_{poly}^{\PP^1_{4,4,2}}(t_0,\tau,t_{i,j}, x^\infty(\tau), y^\infty(\tau), z^\infty(\tau), w^\infty(\tau)).
    $$
    Namely $F_{an}^{\PP^1_{4,4,2}}(t_0,\tau,t_{i,j})$ is obtained by substituting $t_{-1} = \tau$, $x^\infty(\tau)$ instead of $x(q)$ and so on.
  \end{notation}

  \begin{proposition}
    The function $F_{an}^{\PP^1_{4,4,2}}(\tau)$ is holomorphic on $\CC^8\times\HH$ and is a solution to WDVV equation.
  \end{proposition}
  \begin{proof}
    The proof is straightforward
  \end{proof}
  
  It's important to note that we can write the function $F_{an}^{\PP^1_{4,4,2}}$ via the functions $X_k^\infty(\tau)$ too by using the following formulae.
  \begin{equation}\label{eq: xyz via HalpT}
    \begin{aligned}
    x^\infty(\tau) &=  \frac{1}{2} \left(\sqrt{\left(X^\infty_2(\tau) -X^\infty_4(\tau) \right)} + \sqrt{\left(X^\infty_2(\tau) -X^\infty_3(\tau) \right)}\right),
    \\
    y^\infty(\tau) &=  \frac{1}{2} \left(\sqrt{\left(X^\infty_2(\tau) -X^\infty_4(\tau) \right)} -\sqrt{\left(X^\infty_2(\tau) -X^\infty_3(\tau) \right)} \right),
    \\
    z^\infty(\tau) &=  \sqrt{\left(X^\infty_3(\tau) -X^\infty_4(\tau) \right)},
    \\
    w^\infty(\tau) &=  \frac{1}{4}\Bigg(2 X^\infty_2(\tau) +X^\infty_3(\tau) + X^\infty_4(\tau) 
    \\
    &\quad\quad\quad  +2 \sqrt{ \left(X^\infty_2(\tau) -X^\infty_3(\tau) \right)\left(X^\infty_2(\tau) -X^\infty_4(\tau) \right)}\Bigg),
    \end{aligned}
  \end{equation}
  where we choose the square root branch as for $x(q),y(q),z(q),w(q)$ in Eq.~\eqref{eq: xyz via Halphen} by using relation of Notation~\ref{notation: X_4}.

\section{Group actions of the space of genus CohFT potentials}\label{section: group actions}
  For a fixed state space $(V,\eta)$, consider the space of all CohFTs on it. On this space there is a group action, called \textit{Givental's action}, or \textit{upper--triangular} group action. This was first proposed by Givental \cite{G} in genus zero and later developed by the other researchers in the higher genera \cite{S,FSZ}.
  
  The upper--triangular group is defined to be $\{ R \in \mathrm{End}(V) [[z]] \mid R(z)R(-z)^T = 1 \}$. To its element $R = \exp(r(z))$ one can associate the differential operator $\hat R$, s.t. for any CohFT partition function $\Z$ on $(V,\eta)$, the function $\Z' := \hat R \cdot \Z$, is a partition function of a CohFT on the same state space. The action of the upper--triangular group element is also called \textit{R--action} of Givental.
  
  Similarly to the upper--triangular group, one can consider the action of the \textit{lower--triangular} group := $\{ S \in \mathrm{End}(V) [[z^{-1}]] \mid S(z)S(-z)^T = 1 \}$. The action of this group on a CohFT partition function is equivalent to the linear change of the variables, and probably, addition of some new terms to $\F_0$. The action of the lower--triangular group element is also called \textit{S--action} of Givental. We will denote the $S$--action by $\hat S$.
  
  Givental's action appeared to be a powerful tool in working with the CohFTs last decades. However it's usually hard to compute (namely, to give the function $\tilde \Z := \hat R \cdot \Z$ is a closed form). At the same time, there are the situations, when the other action can be introduced, acting on the smaller space, compared to the Givental's action. Being not that general as Givental's action, it can, however make use of some properties, that are specific for this smaller class of CohFTs.   
  In what follows we will work with this sort of actions.
  
  Finally we formulate our results in terms of Givental's action, as playing de facto the role of a canonical group action on the space of CohFT partition functions.

  \subsection{$\SL2C$--group action on the potentials of elliptic orbifolds}
    Consider a unital CohFT on the state space $(V,\eta)$, s.t. $V = \langle e_1,\dots,e_n\rangle$, the unit vector is $e_1$ and $\eta_{1,\alpha} = \delta_{\alpha,n}$.
    Then $F_0(\bt)$, the primary genus $0$ potential, reads:
    $$
      F_0(t_1,\dots,t_n) = \frac{t_1^2t_n}{2} + t_1 \sum_{1 < \alpha \le \beta < n} \eta_{\alpha,\beta} \frac{t_\alpha t_\beta}{|\mathrm{Aut}(\alpha,\beta)|} + H(t_2,\dots,t_n),
    $$
    where $|\mathrm{Aut}(\alpha,\beta)| = 2$ if $\alpha = \beta$ and $1$ otherwise.
    
    For any $A \in \SL2C$ consider another function $F_0^A = F_0^A(t_1,\dots,t_n)$.
    \begin{equation}\label{eq:SLAction}
      \begin{aligned}
	F_0^A \left(t_1,\dots,t_n \right) & := \frac{t_1^2t_n}{2} + t_1 \sum_{1 < \alpha \le \beta < n} \eta_{\alpha,\beta} \frac{t_\alpha t_\beta}{|\mathrm{Aut}(\alpha,\beta)|} + \frac{c\left( \sum_{1 < \alpha \le \beta < n} \eta_{\alpha,\beta} \frac{t_\alpha t_\beta}{|\mathrm{Aut}(\alpha,\beta)|} \right)^2}{2(ct_n+d)}
	\\
	& + (ct_n +d)^2 H \left(\frac{t_2}{ct_n + d},\dots,\frac{t_{n-1}}{ct_n + d}, \frac{at_n + b}{ct_n + d} \right)
	\quad
	\text{for}
	\quad
	A = \begin{pmatrix} 
	a & b \\ c & d
	\end{pmatrix}.
      \end{aligned}
    \end{equation}
    It's not hard to see that $F_0^A$ is solution to WDVV equation and hence a genus $0$ primary potential of some CohFT. 

    It was shown in \cite{B} that the $\SL2C$--action $F_0 \to F_0^A$ can be written via the Givental's R--action. In what follows for any CohFT partition function $\Z$ and any Givental's upper-- or lower--triangular group element $X$ we use the notation 
    $$
      \hat X \cdot F_0 := \mathrm{res}_{\hbar} \left( \hat X \cdot \Z \right) 
    $$ 
    where $F_0 = \mathrm{res}_{\hbar} \left( \Z \right)$. This notation can also be supported by the fact that only genus zero correlators of the initial CohFT contribute to the genus zero correlators of the Givental--transformed CohFT.
    
    For a function $f(\bt)$ we denote by $\left(f(\bt)\right)_p$ the expansion of it at the point $\bt = p$.
    \begin{theorem}[Theorem~3 and Section~5 in \cite{B}]\label{theorem: SL2C to Givental}
    Fix some
    $A = \begin{pmatrix}
	  a & b \\
	  c & d
	\end{pmatrix}
      \in {\rm SL}(2, \mathbb C)
    $ and $\tau \in \CC$, s.t. $c\tau + d\neq 0$. Fix a CohFT with the primary genus zero potential $F_0(\bt)$.
    Let $F_0(\bt)$ and $F_0^A(\bt)$ be convergent in some small neighborhoods of $p_1 := (0,\dots,0, A \cdot \tau)$ and of $p_2 := (0,\dots,0,\tau)$ respectively. 
    For $\sigma := -c(c\tau+d)$, $\sigma^\prime := -c/(c\tau+d)$ 
    holds:
      \begin{align*}
	\left(F^A_0\right)_{p_2} &=  \left( \hat S_0^A \right)^{-1} \cdot \hat R^{\sigma} \cdot \left( F_0\right){p_1},
 	\\
 	\left(F^A_0\right)_{p_2} &=  \hat R^{\sigma^\prime}  \cdot \left( \hat S_0^A \right)^{-1}\cdot \left( F_0\right){p_1},
      \end{align*}
      where 
      \begin{equation*}
      R^\sigma(z) := 
	\exp( \left(
	\begin{array}{c c c}
	  0 & \dots & \sigma \\
	  \vdots & 0 & \vdots \\
	  0 & \dots & 0
	\end{array}
      \right) z),
      \quad
	S_0^A := 
	\left(
	\begin{array}{c c c}
	  1 & \dots & 0 \\
	  \vdots & (c\tau+d) I_{n-2} & \vdots \\
	  0 & \dots & (c\tau+d)^2
	\end{array}
	\right).
    \end{equation*}
    
\end{theorem}

    The theorem above has an extension to the higher genera too (Theorem~3 in \cite{B}), we just don't give it here because at the moment it doesn't play a role. Note that the expansion of the potential at some point can be viewed as an S--action of Givental.

    In \cite[Theorem~6]{B} it was shown, that the $\SL2C$--action above is equivalent to the primitive form change for the simple--elliptic singularities. Due to this fact we don't need to consider the action of full upper--triangular group for the CY/LG correspondence when assuming simple--elliptic singularities only --- the $\SL2C$--action above is the enough. Big advantage of it is clear from the following sections.
    
  \subsection{$\SL2C$--action on the space of Halphen's system solutions}
    For any $A \in \SL2C$ the triple of functions $\{X_2^A(\tau),X_3^A(\tau),X_4^A(\tau)\}$ defined as follows is a solution to the Halphen's system of equations \eqref{eq: Halp in tau} too\footnote{this can be easily checked by hands}.
    \begin{equation}\label{eq: SL action on Halp}
      X_k^A(\tau) := \frac{1}{(c \tau + d)^2} X_k^\infty \left( \frac{a \tau + b}{c \tau + d} \right) + \frac{c}{c \tau + d}, 
      \quad 
      A = \begin{pmatrix}
	a & b \\ c & d
      \end{pmatrix}.
    \end{equation}
    Recall that the analytic genus zero GW potentials of $\PP^1_{4,4,2}$ and $\PP^1_{2,2,2,2}$ are written via the functions $X_k^\infty(\tau)$, and the WDVV equation on them is equivalent to the Halphen's system of equations. Consider the new functions:
    \begin{align*}
      & A \cdot F_{an}^{\PP^1_{2,2,2,2}} := F_{an}^{\PP^1_{2,2,2,2}} \mid_{ \left[ \{X_2^\infty,X_3^\infty, X_4^\infty\} \to \{X_2^A,X_3^A, X_4^A\} \right]},
      \\
      & A \cdot F_{an}^{\PP^1_{4,4,2}} := F_{an}^{\PP^1_{4,4,2}} \mid_{ \left[ \{X_2^\infty,X_3^\infty, X_4^\infty\} \to \{X_2^A,X_3^A, X_4^A\} \right]},
    \end{align*}
    obtained by substituting one solution to the Halphen's system $\{X_2^\infty,X_3^\infty, X_4^\infty\}$ by the other $\{X_2^A,X_3^A, X_4^A\}$. These functions will also be solutions to the WDVV equation and define the same pairing as the previous two.

    The following proposition connects the $\SL2C$--action of Eq.~\eqref{eq:SLAction} (on the space of WDVV equation solutions) with the $\SL2C$--action of Eq.~\eqref{eq: SL action on Halp} (on the space of Halphen's equation solutions).
    \begin{proposition}\label{proposition: SL action general to Halp}
      For any $A \in \SL2C$, the action of it on $F_{an}^{\PP^1_{2,2,2,2}}$ and $F_{an}^{\PP^1_{4,4,2}}$ via Eq.\eqref{eq:SLAction} is equivalent to the action of $A$ on the triple $\{X_2^\infty,X_3^\infty, X_4^\infty\}$ as is Eq.\eqref{eq: SL action on Halp}:
      \begin{align*}
	\left( F_{an}^{\PP^1_{2,2,2,2}} \right)^A &= F_{an}^{\PP^1_{2,2,2,2}} \mid_{ \left[ \{X_2^\infty,X_3^\infty, X_4^\infty\} \to \{X_2^A,X_3^A, X_4^A\} \right]},
	\\
	\left( F_{an}^{\PP^1_{4,4,2}} \right)^A &= F_{an}^{\PP^1_{4,4,2}} \mid_{ \left[ \{X_2^\infty,X_3^\infty, X_4^\infty\} \to \{X_2^A,X_3^A, X_4^A\} \right]}
      \end{align*}
    \end{proposition}
    \begin{proof}
      This is easy to see from the explicit form of the potential $F_0^{\PP^1_{4,4,2}}$ (see Appendix~\ref{section: appendix GW}), Eq.~\eqref{eq: xyz via Halphen} and Proposition~\ref{prop: gw potential basics}.
      
      In particular for the first step we see that the functions $x^\infty(\tau)$, $y^\infty(\tau)$, $z^\infty(\tau)$ only get the factor of $(c \tau +d)^{-1}$ if one substitutes $X_k^\infty$ by $X_k^A$ while the function $w^\infty(\tau)$ gets indeed an additional summand of $c/(c \tau +d)$. For the second step we note that the functions $x^\infty$, $y^\infty$, $z^\infty$ come to the potential so that the factor of $(c \tau +d)^{-1}$ matches the formula of Eq.~\eqref{eq:SLAction} by Proposition~\ref{prop: gw potential basics}. And for the last step we note that this is only the function $w^\infty(\tau)$, that appears with the factor of $t_it_jt_kt_l$ s.t. $\eta(\p_{t_k}, \p_{t_l}) \eta(\p_{t_i}, \p_{t_j}) \neq 0$. Hence the additional summand it gets corresponds exactly to the additional summand of Eq.~\eqref{eq:SLAction}. 
    \end{proof}

    Due to this proposition we will use the notations $A \cdot F$ and $F^A$ without making difference between them.
    
    \begin{notation}\label{notation: xAyAzAwA}
    For any $A \in \SL2C$ denote by $x^A(\tau)$,$y^A(\tau)$,$z^A(\tau)$ and $w^A(\tau)$ the functions obtained from $x^\infty(\tau)$,$y^\infty(\tau)$,$z^\infty(\tau)$ and $w^\infty(\tau)$ by the substitution of the proposition above as in Eq.~\eqref{eq: xyz via HalpT}.
    \end{notation}

    The following proposition makes the connection between the $\SL2C$--actions on $F_{an}^{\X_k}$ and $F_0^{\X_k}$(see also Proposition~4.6 in \cite{BP}).
    \begin{proposition}\label{proposition: sl Fan to F}
      For any $A \in \SL2C$ consider the genus zero potential $F_0^{\X_k} = F_0^{\X_k}(\bt)$ of $\X_k$ written in the formal variables $\bt$  and the analytic potential $F_{an}^{\X_k}(\tau)$.
      Let $\lambda_k = \sqrt{2\pi\sqrt{-1}/k}$ be as in Notation~\ref{notation: X_4}. The following relation holds:
      $$
	A \cdot F_{an}^{\X_k}(\tau) = \left( A' \cdot F_0^{\X_k}(\bt) \right)\mid_{t_{-1} = \tau}.
      $$
      where for 
      $
	A = 
	\begin{pmatrix}
	  a & b
	  \\
	  c & d
	\end{pmatrix}$, we set $
	A' := 
	\begin{pmatrix}
	  a\lambda_k & b\lambda_k
	  \\
	  c\lambda_k^{-1} & d\lambda_k^{-1}
	\end{pmatrix}.
      $
    \end{proposition}
    \begin{proof}
      This follows immediately from the explicit form of the action and Proposition~\ref{proposition: SL action general to Halp} above.
    \end{proof}

    \subsection{The action of $\A^{(\tau_0,\omega_0)}$}
    In what follows we will be in particular interested in the action of the $\SL2C$ elements of the certain form.
    For any fixed $\tau_0 \in \HH$, $\omega_0 \in \CC^*$ define:
    $$
      \A^{(\tau_0, \omega_0)} := 
      \begin{pmatrix}
	  \dfrac{\sqrt{-1}\bar{\tau}_0}{2\omega_0{\rm Im}(\tau_0)} & \omega_0 \tau_0
	  \\
	  \dfrac{\sqrt{-1}}{2\omega_0{\rm Im}(\tau_0)} & \omega_0
	\end{pmatrix} \in \SL2C.
    $$
    This special choice of a $\SL2C$ element comes from singularity theory assumptions and was first proposed\footnote{note however that in the reference given this element was introduced to have $\det = 1/(2\pi\sqrt{-1})$ for any $\tau_0$ and $\omega_0$. We rescale it here because we want to work with the $\SL2C$ element} in \cite{BT}. It has a special meaning in our treatment and we will comment on it later.

    \begin{notation}\label{notation: Xtauomega}
    For any any fixed $\tau_0 \in \HH$, $\omega_0 \in \CC^*$ by using Eq.~\eqref{eq: SL action on Halp} denote: 
    $$
      X_k^{(\tau_0,\omega_0)}(t) := \left( X_k^\infty(t) \right)^{\A^{(\tau_0, \omega_0)}}, \quad 2 \le k \le 4.
    $$
    \end{notation}
    It's easy to see that the functions $X_k^{(\tau_0,\omega_0)}(t)$ are holomorphic in $\lbrace t \in \CC \ | \ |t|<|2\omega_0{\rm Im}(\tau_0)| \rbrace$.
  
  \section{CY/LG correspondence}\label{section: CYLG}
    The idea of CY/LG correspondence came from global Mirror symmetry conjecture. In its framework both FJRW theory and GW theory appear to be the A--side models. The B--model of the global mirror symmetry is given by a singularity with a symmetry group fixed. However it should be understood globally, as varying in a family, given by the different choices of an additional structure --- {\it primitive form} of the singularity. On the B--side, different choices of the primitive form should give (generally) different CohFTs, understood as \textit{different phases} of the one B--model.
    
    The A--model is said to be mirror to the B--model if the partition function of the A--model CohFT coincides up to an S--action of Givental with the partition function of the B--model with some primitive form choice.
    It can happen that two A--models are mirror to the same B--model (taken in the different phases). Then two mirror B--model partition functions differ by a primitive form change. This led to the conjecture, that there should be a R--action of Givental, connecting two B--model CohFTs of the same singularity with the different primitive form choice, or, up to a mirror symmetry equivalently, there should be a R--action of Givental, connecting two A--models, that are mirror to the same global B--model.

    Another important aspect of the global mirror symmetry is the symmetry group, that should be present on both A and B sides.
    Namely, everything said above should hold in the equivariant setting, when both A--model and B--model are considered with some symmetry groups. This is now ultimately realized on the A--side (by FJRW theory in particular), but missing in full generality on the B--side (see \cite{BTW1,BTW2}).

    In \cite{BT} the action of $\A^{(\tau_0,\omega_0)}$ was considered as a \textit{model} for the primitive form change for simple--elliptic singularities. Even as there is no construction of the orbifolded B--model CohFT, one can use the action $\A^{(\tau_0,\omega_0)}$, standing (conjecturally, being equivalent) for the primitive form change of the orbifolded B--model.
    The results of this paper support this conjectural usage of it.
    
  \subsection{Simple--elliptic singularities with the maximal symmetry group}\label{section: SES Gmax}
    The global mirror symmetry program conjectures that for the B--model with the trivial symmetry group, the symmetry group of the A--model should be maximal --- $G_{max}$. In this case the B--model is given by the so--called Saito--Givental CohFT and several different mirror symmetry results were proven (see \cite{CR,MS,MR,KS,LLSS,SZ2,PS,BP}). 
    
    From this variety of mirror symmetry results, in this paper the most important for us is the following $G_{max}$---CY/LG correspondence theorem. Let the basis of $H^*_{orb}(\PP^1_{4,4,2})$ be as in Section~\ref{section: GW theory} and $\lambda_4$ be as in Notation~\ref{notation: X_4}.

    \begin{theorem}[Theorem~4.1 and Lemma~4.9 in \cite{BP}]\label{theorem: CYLG Gmax}
      Consider the FJRW theory of the pair $(\tilde E_7,G_{max})$ and the GW theory of $\PP^1_{4,4,2}$. We have:
      $$
	F_0^{(\tilde E_7, G_{max})}(\tilde \bt) = \A^{(\tau_0,\omega_0)} \cdot F_{an}^{\PP^1_{4,4,2}}(\bt),
      $$
      for $\tau_0 = \sqrt{-1}$, $\omega_0 = \lambda_4 \sqrt{2\pi}/ \left( \Gamma (3/4) \right)^2$ and the certain linear change of variables $\tilde \bt = \tilde \bt(\bt)$. Moreover for the upper--triangular group element $R^{\sigma'}$:
      $$
	R^{\sigma^\prime} := 
	  \exp( \left(
	  \begin{array}{c c c}
	    0 & \dots & \sigma' \\
	    \vdots & 0 & \vdots \\
	    0 & \dots & 0
	  \end{array}
	\right) z), \quad \text{ where } \quad \sigma' = -\frac{1}{2\pi^2}\left(\Gamma(\frac{3}{4})\right)^4,
      $$
      up to the certain $S$--action holds:
      $$
	F_0^{(\tilde E_7, G_{max})} = \hat R^{\sigma'} \cdot \hat S \cdot F_0^{\PP^1_{4,4,2}}.
      $$
    \end{theorem}

    The change of the variables $\tilde \bt(\bt)$ is the following one. We need first to fix the basis in FJRW theory of $(\tilde E_7,G_{max})$. For $W = x^4 + y^4 + z^2$ we have $G_{max} = \langle \rho_1,\rho_2,\rho_3\rangle$, where $\rho_1(x,y,z) = (-\sqrt{-1}x,y,z)$, $\rho_2(x,y,z) = (x,-\sqrt{-1}y,z)$ and $\rho_3(x,y,z) = (x,y,-z)$. The basis of $\H_{\tilde E_7, G_{max}}$ can then be written as $\{[\rho_1^{i}\rho_2^{j}\rho_3,1]\}$ for $1 \le i,j \le 3$.
    
    The change of the variables reads:
    \begin{align*}
      & t_{1,1} = \sqrt{-1}\sqrt{2}\left(\widetilde t_{\rho_1\rho_2^2\rho_3}-\widetilde t_{\rho_1^2\rho_2\rho_3}\right), 
      \ t_{1,2} = -\widetilde t_{\rho_1\rho_2^3\rho_3} + \sqrt{2} \ \widetilde t_{\rho_1^2\rho_2^2\rho_3}-\widetilde t_{\rho_1^3\rho_2\rho_3}, 
      \\ 
      & t_{1,3} = \sqrt{-1}\sqrt{2}\left(\widetilde t_{\rho_1^2\rho_2^3\rho_3}-\widetilde t_{\rho_1^3\rho_2^2\rho_3}\right), 
      \ 
      t_{2,1} = \sqrt{2} \left(\widetilde t_{\rho_1\rho_2^2\rho_3} + \widetilde t_{\rho_1^2\rho_2\rho_3}\right), 
      \\ 
      & t_{2,2} = \widetilde t_{\rho_1\rho_2^3\rho_3} + \sqrt{2} \ \widetilde t_{\rho_1^2\rho_2^2\rho_3}+\widetilde t_{\rho_1^3\rho_2\rho_3}, \ t_{2,3} = \sqrt{2} \left(\widetilde t_{\rho_1^2\rho_2^3\rho_3}+\widetilde t_{\rho_1^3\rho_2^2\rho_3}\right), 
      \\
      & t_{3,1} = \sqrt{-1}\left(\widetilde t_{\rho_1\rho_2^3\rho_3}-\widetilde t_{\rho_1^3\rho_2\rho_3}\right), \quad t_0 = \widetilde t_{\rho_1\rho_2\rho_3}, \ t_{-1}= \widetilde t_{\rho_1^3\rho_2^3\rho_3}.
    \end{align*}
    It's not hard to see that this change of the variables is also degree preserving. The S--action of Theorem~\ref{theorem: CYLG Gmax} is given by $\hat S := \hat S^{\tau_0} \cdot \hat S_0$ for $\hat S_0$ being the rescaling of the variables and
    \[
	S^{\tau_0}(z) = 
	\exp \left( \Bigg(
	\begin{array}{c c c}
	  0 & \dots & 0 \\
	  \vdots & 0 & \vdots \\
	  \tau_0 & \dots & 0
	\end{array}\Bigg)
	z^{-1}\right),
    \]
    so that the action of $\hat S^{\tau_0}$ is equivalent to the expansion at the point $t_{-1} = \tau_0$.
    \begin{remark}
      It's important to note, that in the proof \cite[Section~4]{BP} of the theorem above one doesn't use the virtual fundamental cycle of Fan--Jarvis--Ruan, but again only some properties of the FJRW CohFT. It's easy to check that these are only the axioms\ref{subsection: FJRW state space} --- \ref{subsection: FJRW concavity}, we use in this paper, that are used in \cite{BP}.
    \end{remark}
    
    Explicit R--matrix of the theorem above will play a decisive role in the computations we need to perform to prove main theorem of this paper.

    Recall that we can write the function $\A^{(\tau_0,\omega_0)} \cdot F_{an}^{\PP^1_{4,4,2}}$ (and hence $F_0^{(\tilde E_7,G_{max})}$) via the (holomorphic) functions $X^{(\tau_0,\omega_0)}_k$ with $k=2,3,4$.
    For $\tau_0$ and $\omega_0$ as in theorem above the following series expansions hold:
    \begin{align*}
      X_2^{(\tau_0,\omega_0)}(t) &= \frac{1}{4}-\frac{t}{16}+\frac{t^2}{64}-\frac{t^3}{768}+\frac{t^4}{3072}-\frac{t^5}{20480}+\frac{t^6}{245760}-\frac{13 t^7}{20643840}+\frac{t^8}{9175040}+\mathrm{O}\left(t^9\right)
      \\
      X_3^{(\tau_0,\omega_0)}(t) &= \frac{t}{16}-\frac{t^3}{768}+\frac{t^5}{20480}-\frac{13 t^7}{20643840} + \mathrm{O}\left(t^9\right),
      \\
      X_4^{(\tau_0,\omega_0)}(t) &= -\frac{1}{4}-\frac{t}{16}-\frac{t^2}{64}-\frac{t^3}{768}-\frac{t^4}{3072}-\frac{t^5}{20480}-\frac{t^6}{245760}-\frac{13 t^7}{20643840}-\frac{t^8}{9175040}+ \mathrm{O}\left(t^9\right).
    \end{align*}
    We remind also, that these functions have the particular closed formula by Notation~\ref{notation: Xtauomega}, Eq.\eqref{eq: SL action on Halp} and satisfy $X_k^{(\tau_0,\omega_0)} \in \QQ[[t]]$ for all $k=2,3,4$.
        
  \subsection{Simple--elliptic singularities with a non--maximal symmetry group}
    Consider the simple--elliptic singularity $\tilde E_7$ written by $W = x^4 + y^4 + z^2$ and the symmetry groups (recall the notation of Section~\ref{section: FJRW}):
      \begin{align*}
      G_1 &:= \langle a_1,b_1,c_1 \rangle: \ &&a_1(x,y,z) := \left(\sqrt{-1}x, \sqrt{-1}y, z \right), \ b_1(x,y,z) := (x,-y,z),
      \\
      & && c_1(x,y,z) := (x,y,-z),
      \\
      G_2 &:= \langle a_2,b_2\rangle: \ &&a_2(x,y,z)  := \left(\sqrt{-1}x, \sqrt{-1}y, -z \right), \ b_2(x,y,z):= (x,-y,z),
      \\
      G_3 &:= \langle a_3,b_3\rangle: \ &&a_3(x,y,z) : = \left(\sqrt{-1}x, \sqrt{-1}y, z \right), \ b_3(x.y,z) := (x,y,-z),
    \end{align*}
    \begin{theorem}\label{theorem:main}
    Up to the certain different Givental's S--actions $S^{(k)}$ the partition functions of all three FJRW theories $(\tilde E_7,G_1)$, $(\tilde E_7,G_2)$ and $(\tilde E_7,G_3)$ are connected to the partition function of the Gromov--Witten theory of $\PP^1_{2,2,2,2}$ by the same Givental's R--action of:
    $$
      R^{\sigma^\prime} := 
	\exp( \left(
	\begin{array}{c c c}
	  0 & \dots & \sigma' \\
	  \vdots & 0 & \vdots \\
	  0 & \dots & 0
	\end{array}
      \right) z), \quad \text{ for } \quad \sigma' = -\frac{1}{2\pi^2}\left(\Gamma(\frac{3}{4})\right)^4,
    $$
    so that holds:
    $$
      {\mathcal Z}^{(\tilde E_7, G_k)} =  \hat R^{\sigma'}  \cdot \hat S^{(k)} \cdot \Z^{\PP^1_{2,2,2,2}}, \quad k=1,2,3.
    $$
    \end{theorem}
      \begin{proof}
	We show in Propositions~\ref{proposition: G_1 explicit}, \ref{proposition: G_2 explicit} and \ref{proposition: G_3 explicit} of the next section that there are $A_k \in \SL2C$ for $k=1,2,3$, s.t. $F_0^{(\tilde E_7,G_k)} = A_k \cdot F_0^{\PP^1_{4,4,2}}$, acting as in Eq.\eqref{eq:SLAction}. By using topological recursion relation in genus zero together with Theorem~\ref{theorem: SL2C to Givental} we get an $R$--action of Givental, s.t. $\F_0^{(\tilde E_7,G_k)} = \textrm{res}_{\hbar} ( \hat R \cdot S^{(k)} \cdot \Z_0^{\PP^1_{4,4,2}} )$. It turns out that even though the matrices $A_k$ are not the same in all three cases, the $R$--action appears to be the same (however the $S$--actions needed are anyway different).

	The conditions of Theorem~\ref{theorem: SL2C to Givental} require also certain analyticity of the potentials. We know that this holds because of the particular form of $F_0^{\PP^1_{4,4,2}}$ and $X_k^{(\tau_0,\omega_0)}$. Namely, we utilize the fact that Jacobi theta constants and their logarithmic derivatives are holomorphic in $\HH$.
	
	The FJRW theories of $(\tilde E_7,G_k)$ are all semisimple. One can show it for all three functions $F_0^{(\tilde E_7,G_k)}$ by using the explicit expressions of the potentials. In particular the point $\bt = 0$ is not semisimple, however the point in the neighborhood is semisimple, and this is enough because the property of being semisimple is open. It's a computational exercise to see that the point $\bt = (0,1,2,3,-1,0)$ is semisimple for $\A^{(\tau_0,\omega_0)} \cdot F_{an}^{\PP^1_{2,2,2,2}}$.
	We can apply the reconstruction theorem of Teleman \cite{T}, that gives us that our genus zero equality extends to the higher genera too, what completes the proof.
      \end{proof}
      Note that applying Theorem~\ref{theorem: SL2C to Givental} we made a choice, in which order to apply the S and R--actions. In the equality of two partition functions this is equivalent to the choice, on which side to apply the S--action --- on the FJRW, or on the GW side. The S--action used makes a shift of the coordinates. Hence, in order to have the correlators and make the equality of the partition functions reasonable we should have some analyticity statement about the partition function, to which the S--action is applied. We know such a property only on the GW side, what supports the choice made.

    \begin{remark}
    For the particular values of $\tau_0$ and $\omega_0$ as in Theorem~\ref{theorem: CYLG Gmax} and Theorem~\ref{theorem:main}, we have $X_k^{(\tau_0,\omega_0)} \in \QQ[[t]]$ for all $k=2,3,4$. This is indeed a rare situation (see \cite{Bth}), making the potential 
    \end{remark}
  \section{Computations in FJRW theory}\label{section: proofs}
  
  We first reconstruct explicitly the genus primary potentials of the three FJRW theories in question.
  The reconstruction procedure is always the following. We compute the state space of the FJRW theory and write down the genus $0$ potential via the unknown functions, that are restricted by the selection rule, degree axiom and $G_{max}$--invariance axioms. On the next step we identify those unknown functions that are in the concave sector and hence can be taken from the $G_{max}$--FJRW theory by Corollary~\ref{corollary: concave sectors}. The remaining unknown functions are further reconstructed by the WDVV equation.
  
  Note that usually setting up some mirror symmetry isomorphism one doesn't compute genus zero potentials completely. This is because there is usually a small number of correlators, that reconstruct genus zero potential unambiguously by WDVV equation. The steps outlined above force us to work indeed with the genus zero potentials, and not just some coefficients of their series expansions. 
    
  The most amazing example of the reconstruction procedure we perform is the last one, where the concave sector gives only one function we know explicitly out of the total $10$ building up the potential.

  \subsection{Notations}
  In this section we assume $\tau_0$ and $\omega_0$ to be fixed as in Theorem~\ref{theorem: CYLG Gmax}.
  Recall also Notation~\ref{notation: xAyAzAwA} for $x^A$, $y^A$, $z^A$ and $w^A$. We keep:
    \begin{align*}
      \rx_0 &:= x^{\A^{(\tau_0,\omega_0)}}(t), \ \ry_0 := y^{\A^{(\tau_0,\omega_0)}}(t),  \ \rz_0 := z^{\A^{(\tau_0,\omega_0)}}(t),  \ \rw_0 := w^{\A^{(\tau_0,\omega_0)}}(t).
    \end{align*}
  We make use of the several technical lemmas, that are given in Appendix~\ref{appendix A}.
  
  In this section we write the polynomial $W$ in the $\CC$--coordinates $x,y,z$, rather then $x_1,x_2,x_3$, to reduce the number of subscripts appearing.
  
  We also employ the following notation. All $g \in G_{max}$ are represented by the triples $(\alpha_1,\alpha_2,\alpha_3)$, s.t. 
  \[
    g (x,y,z) = (e[\alpha_1] \cdot x, \ e[\alpha_2]\cdot y, \ e[\alpha_3]\cdot z), \quad \alpha_k \in \QQ \cap [0,1).
  \]
  Such set of the rational number is unique for any $g$.
  
  Recall that the term \textit{WDVV equation} denotes the \textit{system} of PDEs~\eqref{eq: wdvv} for all indices $i,j,k,l$. Due to the complicated variable numbering we will say that the particular PDE~\eqref{eq: wdvv} with some $\{i,j,k,l\}$ is fixed by the quadruple $\{t_i,t_j,t_k,t_l\}$.

\subsection{Case 1: $1$--dimensional broad sector}
Consider $W = x^4 + y^4 + z^2$ and the symmetry group $G_1 := \langle a, b, c \rangle$, where $a = (1/4,1/4,0)$, $b = (0,1/2,0)$ and $c = (0,0,1/2)$. We have $ac = J \in G_1$ and $a^2J = J^{-1}$. The state space $\H$ has the following basis:
$$
  \H = \left\lbrace [J,1],[aJ,1],[bJ,1],[a^2bJ,1],[c,xy],[a^2J,1] \right\rbrace.
$$
By using the selection rule and degree axiom the genus $0$ potential of the FJRW -- theory $(\tilde E_7,G_1)$ reads:

\begin{align*}
  F_0^{(\tilde E_7, G_1)}& = \frac{1}{2} t_J^2 t_{a^2 J}+t_J \left(\frac{t_{aJ}^2}{2}+t_{bJ} t_{a^2 bJ}+\frac{t_{c,xy}^2}{32}\right)+t_{c,xy}^4 g_1( t_{a^2 J} )+t_{bJ} t_{a^2 bJ} t_{c,xy}^2 g_2( t_{a^2 J} )
  \\
   +t_{aJ} & t_{a^2 bJ}^2 t_{c,xy} g_3( t_{a^2 J} )+t_{aJ} t_{bJ}^2 t_{c,xy} g_4( t_{a^2 J} )+t_{aJ}^2 t_{c,xy}^2 g_5( t_{a^2 J} )+t_{a^2 bJ}^2 t_{c,xy}^2 h_1( t_{a^2 J} )
  \\
  +t_{bJ}^2 & t_{c,xy}^2 h_2( t_{a^2 J} )+t_{aJ} t_{c,xy}^3 h_3( t_{a^2 J} )+t_{aJ} t_{bJ} t_{a^2 bJ} t_{c,xy} h_4( t_{a^2 J} )
  \\
  +t_{aJ}^3 & t_{c,xy} h_5( t_{a^2 J} )+t_{bJ} t_{a^2 bJ}^3 f_{0,1}( t_{a^2 J} )+t_{bJ}^3 t_{a^2 bJ} f_{0,2}( t_{a^2 J} )+t_{aJ}^2 t_{a^2 bJ}^2 f_{0,3}( t_{a^2 J} )+t_{aJ}^2 t_{bJ}^2 f_{0,4}( t_{a^2 J} )
  \\
  + &t_{a^2 bJ}^4 f_{1,1}( t_{a^2 J} )+t_{bJ}^2 t_{a^2 bJ}^2 f_{1,2}( t_{a^2 J} )+t_{bJ}^4 f_{1,3}( t_{a^2 J} )+t_{aJ}^2 t_{bJ} t_{a^2 bJ} f_{1,4}( t_{a^2 J} )+t_{aJ}^4 f_{1,5}( t_{a^2 J} ).
\end{align*}
for some unknown functions $g_k(t)$, $h_k(t)$ and $f_{1,k}(t)$. However from the selection rule~\ref{subsection: FJR selection rule} we know that all functions $g_k(t)$ are odd while the functions $h_k(t)$ are even.
The correlators of the $(\tilde E_7, G_1)$ theory involving narrow insertions only are concave. Hence we can identify some of the functions above with those from $(\tilde E_7, G_{max})$ -- theory. We have:
\begin{align*}
  & f_{0,1}( t_{a^2 J} ) = 0, \ f_{0,2}( t_{a^2 J} ) = 0, \ & f_{0,3}( t_{a^2 J} ) & = -\frac{\rx_0^2}{8}-\frac{\rx_0 \ry_0}{4}-\frac{\ry_0^2}{8},
  \\
  & f_{0,4}( t_{a^2 J} ) = -\frac{\rx_0^2}{8}-\frac{\rx_0 \ry_0}{4}-\frac{\ry_0^2}{8}, \ & f_{1,1}( t_{a^2 J} ) & = -\frac{\rx_0^2}{48}+\frac{\rx_0 \ry_0}{8}-\frac{\ry_0^2}{48},
  \\
  & f_{1,2}( t_{a^2 J} ) = -\frac{\rw_0}{2}+\frac{3 \rx_0^2}{8}-\frac{\rx_0 \ry_0}{4}-\frac{\ry_0^2}{8}, \ & f_{1,3}( t_{a^2 J} ) & = -\frac{\rx_0^2}{48}+\frac{\rx_0 \ry_0}{8}-\frac{\ry_0^2}{48},
  \\
  & f_{1,4}( t_{a^2 J} ) = -\frac{\rw_0}{2}+\frac{\rx_0^2}{4}+\frac{\rx_0 \ry_0}{2}-\frac{\ry_0^2}{4}, \ & f_{1,5}( t_{a^2 J} ) & = -\frac{\rw_0}{8}+\frac{\rx_0^2}{12}-\frac{\ry_0^2}{24},
\end{align*}
where the functions $\rx_0 = \rx_0(t_{a^2 J})$, $\ry_0 = \ry_0(t_{a^2 J})$, $\rz_0=\rz_0(t_{a^2 J})$, $\rw_0=\rw_0(t_{a^2 J})$ are 
given at the beginning of this section.

\subsubsection{The WDVV equation}\label{subsection: G_1 potential}
  Writing the WDVV equation for $F_0^{(\tilde E_7,G_1)}$ we get the following system:
  \begin{align*}
    \rw_0'(t) = \rw_0^2 - \rx_0^4, \ \rx_0'(t) =  \rx_0 \left(\rw_0 - \rx_0^2 + 2 \ry_0^2\right), \ \ry_0'(t) = \ry_0 \left( \rw_0 + \rx_0^2 \right),
  \end{align*}
  and also
  \begin{align*}
    & g_3( t_{a^2 J} ) =  0, \ g_4( t_{a^2 J} ) =  0, \ h_3( t_{a^2 J} ) =  0, \ h_4( t_{a^2 J} ) =  0, \ h_5( t_{a^2 J} ) =  0,
    \\
    & g_1( t_{a^2 J} ) =  \frac{\rx_0^2}{3072} - \frac{\rw_0}{2048} -\frac{\ry_0^2}{6144}, \ g_2( t_{a^2 J} ) =  \frac{\rx_0^2}{64} +\frac{\rx_0 \ry_0}{32} -\frac{\ry_0^2}{64} -\frac{\rw_0}{32}, 
    \\ 
    & g_5( t_{a^2 J} ) =  - \frac{\rx_0 \ry_0}{32} +\frac{\rx_0^2}{64} -\frac{\rw_0}{64},
    \\
    &h_1( t_{a^2 J} ) =  \frac{\rx_0^2}{128} +\frac{\rx_0 \ry_0}{64} +\frac{\ry_0^2}{128}, \ h_2( t_{a^2 J} ) =  \frac{\rx_0^2}{128} + \frac{\rx_0 \ry_0}{64} +\frac{\ry_0^2}{128}.
  \end{align*}
  In particular it's obtained by taking Eq.~\eqref{eq: wdvv} fixed by the indices: 
  \begin{align*}
     & \Big\{t_{aJ},t_{aJ},t_{bJ},t_{bJ}\Big\} ,  \Big\{t_{aJ},t_{aJ},t_{a^2 bJ},t_{a^2 bJ}\Big\} ,  \Big\{t_{bJ},t_{bJ},t_{c,xy},t_{c,xy}\Big\} ,  \Big\{t_{aJ},t_{aJ},t_{bJ},t_{bJ}\Big\} ,  
     \\
     & \Big\{t_{aJ},t_{aJ},t_{a^2 bJ},t_{a^2 bJ}\Big\} ,  \Big\{t_{aJ},t_{bJ},t_{c,xy},t_{bJ}\Big\} ,  \Big\{t_{c,xy},t_{c,xy},t_{a^2 bJ},t_{a^2 bJ}\Big\},  \Big\{t_{c,xy},t_{c,xy},t_{aJ},t_{aJ}\Big\}.
  \end{align*}
  
  The differential part of the system above involves only the functions we know already and the PDEs written are equivalent to the WDVV of the genus $0$ GW potential of $\PP^1_{4,4,2}$ (see Section~\ref{section: GW theory}). Hence we do not have to solve the PDEs and we know all functions building up $F_0^{(\tilde E_7,G_1)}$ explicitly.
  The potential of this FJRW theory reads:
  \begin{align*}
    & F_0^{(\tilde E_7, G_1)} = \frac{1}{2} t_J^2 t_{a^2 J}+t_J \left(\frac{t_{aJ}^2}{2}+t_{bJ} t_{a^2 bJ}+\frac{t_{c,xy}^2}{32}\right)+ t_{aJ}^4\left(\frac{\rx_0^2}{12}-\frac{\ry_0^2}{24}-\frac{\rw_0}{8}\right) - t_{bJ}^4\left(\frac{\rx_0^2}{48}-\frac{\rx_0 \ry_0}{8}+\frac{\ry_0^2}{48}\right)
    \\
    & - t_{a^2 bJ}^4 \left(\frac{\rx_0^2}{48} -\frac{\rx_0 \ry_0}{8} +\frac{\ry_0^2}{48}\right) + t_{bJ} t_{a^2 bJ} t_{c,xy}^2\left(\frac{\rx_0^2}{64}+\frac{\rx_0 \ry_0}{32}-\frac{\ry_0^2}{64} -\frac{\rw_0}{32}\right) 
    \\
    & +t_{a^2 bJ}^2 t_{c,xy}^2\left(\frac{\rx_0^2}{128}+\frac{\rx_0 \ry_0}{64}+\frac{\ry_0^2}{128}\right) + t_{c,xy}^4\left(\frac{\rx_0^2}{3072}-\frac{\ry_0^2}{6144} -\frac{\rw_0}{2048} \right) 
    \\
    &+t_{aJ}^2 \Bigg[-t_{bJ}^2\left(\frac{\rx_0^2}{8}+\frac{\rx_0 \ry_0}{4}+\frac{\ry_0^2}{8}\right) + t_{bJ} t_{a^2 bJ}\left(\frac{\rx_0^2}{4}+\frac{\rx_0 \ry_0}{2}-\frac{\ry_0^2}{4}-\frac{\rw_0}{2}\right) -t_{a^2 bJ}^2\left(\frac{\rx_0^2}{8}+\frac{\rx_0 \ry_0}{4}+\frac{\ry_0^2}{8}\right)  
    \\
    & +t_{c,xy}^2\left(\frac{\rx_0^2}{64}-\frac{\rx_0 \ry_0}{32}-\frac{\rw_0}{64}\right)  \Bigg]
    +t_{bJ}^2 \left[ t_{a^2 bJ}^2\left(\frac{3 \rx_0^2}{8}-\frac{\rx_0 \ry_0}{4}-\frac{\ry_0^2}{8}-\frac{\rw_0}{2}\right) +t_{c,xy}^2\left(\frac{\rx_0^2}{128}+\frac{\rx_0 \ry_0}{64}+\frac{\ry_0^2}{128}\right) \right].
  \end{align*}
  By using Eq.\eqref{eq: xyz via HalpT} and the definition of the $\A^{(\tau_0,\omega_0)}$--action we get the following proposition.
  \begin{proposition}\label{proposition: primary potential of W,G_1}
    The genus zero primary potential of the FJRW theory of $(\tilde E_7,G_1)$ reads:
  \begin{equation}\label{eq: G_1 FM potential}
    \begin{aligned}
      F_0^{(\tilde E_7, G_1)} &= \frac{1}{2} t_J^2 t_{a^2 J}+t_J \left(\frac{t_{aJ}^2}{2}+t_{bJ} t_{a^2 bJ}+\frac{t_{c,xy}^2}{32}\right)
      -\Big(\frac{t_{aJ}^4}{24}+\frac{t_{bJ}^4}{48}+\frac{t_{a^2 bJ}^4}{48}
      \\
      & +\frac{t_{c,xy}^4}{6144}+\frac{1}{4} t_{aJ}^2 t_{bJ} t_{a^2 bJ} +\frac{1}{8} t_{bJ}^2 t_{a^2 bJ}^2+\frac{1}{64} t_{bJ} t_{a^2 bJ} t_{c,xy}^2\Big) \left(X_2^{(\tau_0,\omega_0)}  + X_4^{(\tau_0,\omega_0)}  \right)
      \\
      & +\left(\frac{1}{128} t_{bJ}^2 t_{c,xy}^2+\frac{1}{128} t_{a^2 bJ}^2 t_{c,xy}^2-\frac{1}{8} t_{aJ}^2 t_{bJ}^2-\frac{1}{8} t_{aJ}^2 t_{a^2 bJ}^2\right) \left(X_2^{(\tau_0,\omega_0)}  -X_4^{(\tau_0,\omega_0)}  \right)
      \\
      &-\left(\frac{t_{aJ}^4}{24}-\frac{t_{bJ}^4}{24}-\frac{t_{a^2 bJ}^4}{24}+\frac{t_{c,xy}^4}{6144}+\frac{1}{4} t_{bJ}^2 t_{a^2 bJ}^2+\frac{1}{64} t_{aJ}^2 t_{c,xy}^2\right) X_3^{(\tau_0,\omega_0)} ,
    \end{aligned}
  \end{equation}
  where $X_k^{(\tau_0,\omega_0)} = X_k^{(\tau_0,\omega_0)}(t_{a^2 J})$ are as in Section~\ref{section: SES Gmax}.
  \end{proposition}
  
\subsubsection{CY/LG correspondence}
Consider the change of the variables:
\begin{equation}\label{eq: G1 linear change}
\begin{aligned}
  & t_J = t_0, \ t_{a^2 J} = \tau
  \\
  & t_{aJ} = \frac{t_1}{\sqrt{2}}, \ t_{bJ} = \frac{t_2}{2} + \frac{\sqrt{-1} t_3}{2}, \ t_{a^2 bJ} = \frac{t_2}{2}-\frac{\sqrt{-1} t_3}{2}, \ t_{c,xy} = 2 \sqrt{2} t_4.
\end{aligned}
\end{equation}
By using Eq.~\eqref{eq: G_1 FM potential} we get:
  \begin{align*}
    F_0^{(\tilde E_7, G_1)} &= \frac{1}{2} t_0^2 \tau +\frac{1}{4} t_0 \sum_{k=1}^4 t_k^2  
    -\frac{1}{16} \left(t_3^2 t_4^2 + t_1^2 t_2^2 \right) X_2^{(\tau_0,\omega_0)}(\tau)  -\frac{1}{16} \left(t_1^2 t_3^2 + t_2^2 t_4^2 \right) X_4^{(\tau_0,\omega_0)}(\tau) 
    \\
    &  - \frac{1}{16} \left( t_2^2 t_3^2 + t_1^2 t_4^2 \right) X_3^{(\tau_0,\omega_0)}(\tau) - \frac{1}{96} \left( \sum_{k=1}^4 t_k^4 \right) \left( \sum_{k=2}^4 X_k^{(\tau_0,\omega_0)} (\tau) \right).
  \end{align*}
  It's obvious that we get:
  $$
    F_0^{(\tilde E_7, G_1)}(\tilde \bt(\bt))  = \A^{(\tau_0,\omega_0)} \cdot F_{an}^{\PP^1_{2,2,2,2}}.
  $$
  In order to derive the equality for the potential $F_0^{\PP^1_{2,2,2,2}}$ we apply Proposition~\ref{proposition: sl Fan to F}. We proved:
  
  \begin{proposition}\label{proposition: G_1 explicit}
  For the linear change of the variables as above holds:
  $$
    F_0^{(\tilde E_7, G_1)}(\tilde \bt(\bt))  = A^{G_1} \cdot F_0^{\PP^1_{2,2,2,2}}, \quad 
    A^{G_1} := 
      \begin{pmatrix}
	  \dfrac{1}{\Theta} & - \pi \Theta
	  \\
	  \dfrac{1}{2\pi \Theta} & \dfrac{\Theta}{2}
	\end{pmatrix}
  $$
  for $\Theta = \sqrt{2\pi}/\left(\Gamma(\frac{3}{4})\right)^2$.
  \end{proposition}
  
\subsection{Case 2: $2$--dimensional broad sector}
Consider $W = x^4 + y^4 + z^2$ and the symmetry group $G_2 := \langle a, b\rangle$, where $a = (1/4,1/4,1/2)$, $b = (0,1/2,0)$. We have $a = J \in G_2$ and $a^3J = J^{-1}$. The state space $\H$ has the following basis:
$$
  \H = \left\lbrace [J,1],[ab,1],[a^3b,1],[a^2b,xy],[b,x],[a^3,1] \right\rbrace.
$$
By using the selection rule, degree axiom and $G_{max}$--invariance axiom the genus $0$ potential of the FJRW -- theory $(\tilde E_7,G_2)$ reads
\begin{align*}
  & F_0^{(\tilde E_7,G_2)} = \frac{1}{2} t_{a^3} t_J^2+t_J \left(t_{ab} t_{a^3 b}+\frac{t_{b,x}^2}{16}+\frac{1}{32} t_{a^2 b,xy}^2\right)+t_{b,x}^4 g_1(t_{a^3})+t_{b,x}^2 t_{a^2 b,xy}^2 g_2(t_{a^3})+t_{a^2 b,xy}^4 g_3(t_{a^3})
  \\
  & +t_{ab} t_{a^3 b} t_{b,x}^2 g_4(t_{a^3})+t_{ab} t_{a^3 b} t_{a^2 b,xy}^2 g_5(t_{a^3})+t_{a^3 b}^2 t_{b,x}^2 h_1(t_{a^3})+t_{a^3 b}^2 t_{a^2 b,xy}^2 h_2(t_{a^3})+t_{ab}^2 t_{b,x}^2 h_3(t_{a^3})
  \\
  &+t_{ab}^2 t_{a^2 b,xy}^2 h_4(t_{a^3})+t_{ab} t_{a^3 b}^3 f_{0,1}(t_{a^3})+t_{ab}^3 t_{a^3 b} f_{0,2}(t_{a^3})+t_{a^3 b}^4 f_{1,1}(t_{a^3})+t_{ab}^2 t_{a^3 b}^2 f_{1,2}(t_{a^3})+t_{ab}^4 f_{1,3}(t_{a^3}),
\end{align*}
for some unknown functions $g_k(t)$, $h_k(t)$ and $f_{k,l}(t)$. However from the selection rule~\ref{subsection: FJR selection rule} we know that all functions $g_k(t)$ are odd while the functions $h_k(t)$ are even.
The correlators of the $(\tilde E_7,G_2)$ theory involving narrow insertions only are concave. Hence we can identify some of the functions above with those from $(\tilde E_7,G_{max})$ -- theory. We have:
\begin{align*}
  & f_{0,1}(t_{a^3}) = 0, \ f_{0,2}(t_{a^3}) = 0, \ f_{1,1}(t_{a^3}) = -\frac{\rx_0^2}{48} +\frac{\rx_0 \ry_0}{8} -\frac{\ry_0^2}{48},
  \\
  & f_{1,2}(t_{a^3}) = -\frac{\rw_0}{2} + \frac{3 \rx_0^2}{8} - \frac{\rx_0 \ry_0}{4} -\frac{\ry_0^2}{8}, \ f_{1,3}(t_{a^3}) = -\frac{\rx_0^2}{48}+\frac{\rx_0 \ry_0}{8} -\frac{\ry_0^2}{48}.
\end{align*}

\subsubsection{The WDVV equation}\label{section: G_2 solving wdvv}
  Writing the WDVV equation for $F_0^{(\tilde E_7,G_2)}$ it's enough to consider Eq.~\eqref{eq: wdvv} with the parameters 
  \begin{align*}
    \left\{t_{ab},t_{ab},t_{a^3 b},t_{a^2 b,xy}\right\}, \left\{t_{a^2 b,xy},t_{a^2 b,xy},t_{ab},t_{ab}\right\}, \left\{t_{a^3 b},t_{a^3 b},t_{a^2 b,xy},t_{a^2 b,xy}\right\}, 
    \\
    \left\{t_{a^3 b},t_{a^2 b,xy},t_{b,x},t_{ab}\right\}, \left\{t_{b,x},t_{b,x},t_{a^2 b,xy},t_{a^2 b,xy}\right\}, \left\{t_{a^2 b,xy},t_{ab},t_{a^2 b,xy},t_{ab}\right\}.
  \end{align*}
  We get two cases.
  The first one is when $h_2(t) \equiv 0$ or $h_4(t) \equiv 0$. This case also concludes $f_{1,1}(t) \equiv 0$, what we know to be false. For the second case we have $h_2(t)h_4(t) \not\equiv 0$ and the following system should be solved:
  \begin{align*}
    &g_3'(t) = \frac{1}{2} g_5(t)^2 + \frac{2}{3} h_2(t) h_4(t) -64 g_3(t) g_5(t), 
    \\ 
    &g_5'(t) = - 32 g_5(t) ^2 + 128 h_2(t) h_4(t), 
    \\
    &h_2'(t) = 128 \left(g_5(t) - 96 g_3(t)\right) h_2(t), 
    \\ 
    &h_4'(t) = 128 \left(g_5(t) - 96 g_3(t)\right) h_4(t),
  \end{align*}
  and also
  \begin{align*}
    f_{1,1}(t) &= -\frac{h_2(t)}{h_4(t)} \left(256 g_3(t) - 4 g_5(t)\right), \ f_{1,2}(t) =  8 \left( 192 g_3(t) - g_5(t) \right), \ g_1(t) =  4g_3(t),
    \\
    g_2(t) &= 12 g_3(t) -\frac{1}{8} g_5(t), \ g_4(t) = 2 g_5(t), \ h_1(t) = - 2h_2(t), \ h_3(t) = -2 h_4(t),
    \\
    h_4(t) &\not\equiv 0, \ h_4(t) \not\equiv 0.
  \end{align*}
  From the PDEs on $h_2$ and $h_4$ we see that $h_2(t) = c h_4(t)$ for some non--zero complex $c \in \CC$. Hence we get an expression of $g_3(t)$ and $g_5(t)$ via the functions $X_k^{(\tau_0,\omega_0)}$ and the constant $c$. 
  \begin{align*}
    g_3(t) &= -\frac{1}{24576 c} \left((3c+1) X_2^{(\tau_0,\omega_0)}(t) + 2 (3c-1) X_3^{(\tau_0,\omega_0)}(4t) + (3c+1) X_4^{(\tau_0,\omega_0)}(t) \right),
    \\
    g_5(t) &= -\frac{1}{128 c} \left( (c+1) X_2^{(\tau_0,\omega_0)}(t) + 2 (c-1) X_3^{(\tau_0,\omega_0)}(4t) +(c+1) X_4^{(\tau_0,\omega_0)}(t) \right).
  \end{align*}
  However we also have two PDEs on $g_3(t)$ and $g_5(t)$ that give us the compatibility condition:
  $$
    \frac{3}{2}\left(64 g_3(t) g_5(t)-\frac{1}{2} g_5(t)^2 + g_3'(t)\right) = \frac{1}{128}\left(32 g_5(t){}^2+g_5'(t)\right)
  $$
  Putting the explicit expressions of $g_3(t)$ and $g_5(t)$ via the functions $X_k^{(\tau_0,\omega_0)}$ here we get that this condition is satisfied if and only if $c^2 = 1$. Knowing the functions $g_3(t)$ and $g_5(t)$ explicitly we resolve the function $h_2(t)$ as the square root.

  This gives us two solutions to the WDVV equation and consider them both in what follows.
  
  \subsubsection{Positive solution}
  For $c = 1$ we get the following solution to this system:
  \begin{align*}
    & f_1(t) = \frac{1}{48} \left(-X^{(\tau_0,\omega_0)}_2 +2 X^{(\tau_0,\omega_0)}_3 -X^{(\tau_0,\omega_0)}_4 \right), \ f_2(t) = -\frac{1}{8} \left(X^{(\tau_0,\omega_0)}_2   + 2 X^{(\tau_0,\omega_0)}_3  +X^{(\tau_0,\omega_0)}_4 \right),
    \\
    & g_1(t)  =  -  \frac{1}{1536} \left( X^{(\tau_0,\omega_0)}_2 +X^{(\tau_0,\omega_0)}_3 +X^{(\tau_0,\omega_0)}_4 \right), \ g_2(t)  =  -\frac{1}{512} X^{(\tau_0,\omega_0)}_3,
    \\
    & h_1(t)  = h_3(t) =  -\frac{1}{64} \sqrt{\left(X^{(\tau_0,\omega_0)}_2 -X^{(\tau_0,\omega_0)}_4 \right)^2}, \ h_2(t) = h_4(t)  =  \frac{1}{128} \sqrt{\left(X^{(\tau_0,\omega_0)}_2 -X^{(\tau_0,\omega_0)}_4 \right)^2},
    \\
    & g_3(t) =  - \frac{1}{6144} \left( X^{(\tau_0,\omega_0)}_2 + X^{(\tau_0,\omega_0)}_3 + X^{(\tau_0,\omega_0)}_4 \right), \ g_4(t)  =  -\frac{1}{32} \left(X^{(\tau_0,\omega_0)}_2 + X^{(\tau_0,\omega_0)}_4 \right),
    \\
    &g_5(t)  =  -\frac{1}{64} \left(X^{(\tau_0,\omega_0)}_2 +X^{(\tau_0,\omega_0)}_4 \right).
  \end{align*}
  


  \subsubsection{Negative solution} For $c = -1$ we get the following answer.
  \begin{align*}
    & f_1(t)  =  \frac{1}{48} \left(2 X_3^{(\tau_0,\omega_0)} -X_2^{(\tau_0,\omega_0)} -X_4^{(\tau_0,\omega_0)}  \right), \ f_2(t) =  - \frac{1}{8} \left(X_2^{(\tau_0,\omega_0)}  +2 X_3^{(\tau_0,\omega_0)} +X_4^{(\tau_0,\omega_0)}  \right),
    \\
    & g_1(t)  =  - \frac{1}{3072} \left( X_2^{(\tau_0,\omega_0)}  +4 X_3^{(\tau_0,\omega_0)}  +X_4^{(\tau_0,\omega_0)} \right) , \ g_2(t)  =  -\frac{1}{1024} \left( X_2^{(\tau_0,\omega_0)}  +X_4^{(\tau_0,\omega_0)} \right),
    \\
    & h_1(t) = - h_3(t)  = 2 h_4(t) = -2 h_2(t) =    \frac{1}{32} \sqrt{\left(X_2^{(\tau_0,\omega_0)}  -X_3^{(\tau_0,\omega_0)}  \right) \left(X_3^{(\tau_0,\omega_0)}  -X_4^{(\tau_0,\omega_0)}  \right)}, 
    \\
    &g_3 (t) =  -\frac{1}{12288} \left( X_2^{(\tau_0,\omega_0)} +4 X_3^{(\tau_0,\omega_0)} +X_4^{(\tau_0,\omega_0)}\right), \ g_4(t)  =  -\frac{1}{16} X_3^{(\tau_0,\omega_0)}, \ g_5  =  -\frac{1}{32} X_3^{(\tau_0,\omega_0)}.
  \end{align*}
  
  \subsubsection{Comparison of the two solutions} 
  In both ``negative'' and ``positive'' solutions above, some square roots need to be resolved. This makes one more sign choice for both cases. However it's easy to see that this sign choice can be realized as the scaling of the variables $t_{ab}$, $t_{a^3b}$, preserving the cubic terms. Because we make our computation modulo such rescaling here, we can make a particular choice of this square root resolution in both cases.
  
  Let $F_0^+$ and $F_0^-$ be the two primary genus zero potentials given by the ``positive'' and ``negative'' solutions to the WDVV above respectively. We establish the connection between them.
  
  \begin{proposition}\label{prop: G_2 two potentials together}
    Let $F_0^+$ be written in coordinates $t_{g,\phi(\bx)}^+$ and $F_0^-$ be written in coordinates $t_{g,\phi(\bx)}^-$. Then they are connected by the following linear change of the variables:
    \begin{align*}
      & t_{J}^- = K^{-2} t_{J}^+, \ t_{J^{-1}}^- = K^2 t_{J^{-1}}^+,
      \\
      & t_{ab}^- = \frac{(1-\sqrt{-1}) K }{\sqrt{2}}t_{ab}^+, \ t_{a^3 b}^- = \frac{(1+\sqrt{-1}) K}{\sqrt{2}}t_{a^3 b}^+, \ t_{a^2 b,xy}^- = K t_{a^2 b,xy}^+, \ t_{b,x}^- = K t_{b,x}^+,
    \end{align*}
    where $K = e^{\pi \sqrt{-1}/2}$.
  \end{proposition}
  \begin{proof}
    It's enough to compare the 4--point correlators, what in our case amounts to the comparison of the potentials with $X_k^{(\tau_0,\omega_0)}$ evaluated at the point $t = 0$. The rest is straightforward.
  \end{proof}

  \begin{proposition}\label{proposition: primary potential of W,G_2}
    Up to a scaling of the variables the genus zero primary potential of the FJRW theory of $(\tilde E_7,G_2)$ reads:    
  \begin{equation*}
    \begin{aligned}
      F_0^{(\tilde E_7, G_2)} &= t_{ab} t_{a^3 b} t_J+\frac{1}{2} t_{a^3} t_J^2+\frac{1}{16} t_J t_{b,x}^2+\frac{1}{32} t_J t_{a^2 b,xy}^2 -\Big(\frac{t_{ab}^4}{48}+\frac{t_{a^3 b}^4}{48}+\frac{t_{b,x}^4}{1536}
      \\
      & +\frac{t_{a^2 b,xy}^4}{6144}+\frac{1}{8} t_{ab}^2 t_{a^3 b}^2 +\frac{1}{32} t_{ab} t_{a^3 b} t_{b,x}^2+\frac{1}{64} t_{ab} t_{a^3 b} t_{a^2 b,xy}^2\Big) \left(X_2^{(\tau_0,\omega_0)} +X_4^{(\tau_0,\omega_0)}  \right)
      \\
      & +\left(\frac{1}{64} t_{ab}^2 t_{b,x}^2+\frac{1}{64} t_{a^3 b}^2 t_{b,x}^2-\frac{1}{128} t_{ab}^2 t_{a^2 b,xy}^2-\frac{1}{128} t_{a^3 b}^2 t_{a^2 b,xy}^2\right)\left( X_4^{(\tau_0,\omega_0)} -X_2^{(\tau_0,\omega_0)} \right) 
      \\
      & +\left(\frac{t_{ab}^4}{24}+\frac{t_{a^3 b}^4}{24} -\frac{t_{b,x}^4}{1536}-\frac{t_{a^2 b,xy}^4}{6144}-\frac{1}{4} t_{ab}^2 t_{a^3 b}^2-\frac{1}{512} t_{b,x}^2 t_{a^2 b,xy}^2\right) X_3^{(\tau_0,\omega_0)} 
    \end{aligned}
  \end{equation*}
  where $X_k^{(\tau_0,\omega_0)} = X_k^{(\tau_0,\omega_0)} \left(t_{a^3}\right)$ are as in Section~\ref{section: SES Gmax}.
  \end{proposition}
  \begin{proof}
    It's easy to see that Proposition~\ref{prop: G_2 two potentials together} above performs the scaling $X_k^{(\tau_0,\omega_0)}(t) \to \sqrt{-1} \cdot X_k^{(\tau_0,\omega_0)} \left( \sqrt{-1} t \right)$. This can be obviously realized as an S--action of Givental. Together with the previous section we get the proof. 
  \end{proof}

\subsubsection{CY/LG correspondence}
By using explicit expression of all the functions coming to $F_0^+$ via $X_k^{(\tau_0,\omega_0)}(t)$ and applying the following change of variables:
\begin{align*}
 & t_J = t_0, \ t_{a^3} = \tau.
 \\
 & t_{ab} = \frac{1}{2} \left( t_1 + \sqrt{-1} t_2 \right), \ t_{a^3 b} = \frac{1}{2} \left( t_1 - \sqrt{-1} t_2 \right), \ t_{a^2 b,xy} = 2 \sqrt{2} t_3, \ t_{b,x} = 2 t_4.
\end{align*}
we get: 
\begin{align*}
  F_0^{(\tilde E_7,G_2)} &= \frac{1}{2} t_0^2 \tau +\frac{1}{4} t_0 \sum_{k=2}^5 t_k^2  
  -\frac{1}{16} \left(t_3^2 t_4^2 + t_1^2 t_2^2 \right) X_3^{(\tau_0,\omega_0)} (t) - \frac{1}{16} \left(t_1^2 t_3^2 + t_2^2 t_4^2 \right) X_4^{(\tau_0,\omega_0)}(t)
  \\
  &  - \frac{1}{16} \left( t_2^2 t_3^2 + t_1^2 t_4^2 \right) X_2^{(\tau_0,\omega_0)} (t) - \frac{1}{96} \left( \sum_{k=2}^5 t_k^4 \right) \left( \sum_{k=2}^4 X_k^{(\tau_0,\omega_0)}(t)\right).
\end{align*}
  It's obvious that we get:
  $$
    F_0^{(\tilde E_7, G_2)}(\tilde \bt(\bt))  = \A^{(\tau_0,\omega_0)} \cdot F_{an}^{\PP^1_{2,2,2,2}}.
  $$
  In order to derive the equality for the potential $F_0^{\PP^1_{2,2,2,2}}$ we apply Proposition~\ref{proposition: sl Fan to F}. We get:
  
\begin{proposition}\label{proposition: G_2 explicit}
  For the linear change of the variables holds:
  $$
    F_0^{(\tilde E_7, G_2)}(\tilde \bt(\bt))  = A^{G_2} \cdot F_0^{\PP^1_{2,2,2,2}}, \quad 
    A^{G_2} := 
      \begin{pmatrix}
	  \dfrac{1}{\Theta} & - \pi \Theta
	  \\
	  \dfrac{1}{2\pi \Theta} & \dfrac{\Theta}{2}
	\end{pmatrix}
  $$
  for $\Theta = \sqrt{2\pi}/\left(\Gamma(\frac{3}{4})\right)^2$.
\end{proposition}
  
\subsection{Case 3: $3$--dimensional broad sector}
Consider $W = x^4 + y^4 + z^2$ and the symmetry group $G_3 := \langle a, b \rangle$, where $a = (1/4,1/4,0)$ and $b = (0,0,1/2)$. We have $ab = J \in G_3$ and $a^2J = J^{-1}$. The state space $\H$ has the following basis:
$$
  \H = \left\lbrace [J,1],[aJ,1],[b,x^2],[b,xy],[b,y^2],[a^2J,1] \right\rbrace.
$$
By using the selection rule, degree axiom and $G_{max}$--invariance axiom the genus $0$ potential of the FJRW -- theory $(\tilde E_7,G_3)$ reads:
\begin{align*}
F_0^{(\tilde E_7,G_3)} & = \frac{1}{2} t_J^2 t_{a^2 J}+t_J \left(\frac{t_{aJ}^2}{2}+\frac{t_{b,xy}^2}{32}+\frac{1}{16} t_{b,x^2} t_{b,y^2}\right)  +t_{b,y^2}^4 g_1\left(t_{a^2 J}\right)+t_{b,xy}^4 g_2\left(t_{a^2 J}\right)
\\
&+t_{b,x^2} t_{b,xy}^2 t_{b,y^2} g_3\left(t_{a^2 J}\right) +t_{b,x^2}^2 t_{b,y^2}^2 g_4\left(t_{a^2 J}\right)+t_{b,x^2}^4 g_5\left(t_{a^2 J}\right)+t_{aJ}^2 t_{b,xy}^2 g_6\left(t_{a^2 J}\right)
\\
&+t_{aJ}^2 t_{b,x^2} t_{b,y^2} g_7\left(t_{a^2 J}\right)
+t_{aJ} t_{b,xy} t_{b,y^2}^2 h_1\left(t_{a^2 J}\right)+t_{aJ} t_{b,x^2}^2 t_{b,xy} h_2\left(t_{a^2 J}\right) + t_{aJ}^4 f_{1,1}\left(t_{a^2 J}\right),
\end{align*}
for some unknown functions $g_k(t)$, $h_k(t)$ and $f_{1,1}(t)$. However from the selection rule~\ref{subsection: FJR selection rule} we know that all functions $g_k(t)$ and also $f_1(t)$ are odd while the functions $h_k(t)$ are even.

Note that the correlators of $(\tilde E_7,G)$ involving the insertions of $[J,1]$, $[a^2J,1]$ and $[aJ,1]$ only are concave. Hence we have an explicit expression for the function $f_{1,1}$ that we have found in $(\tilde E_7,G_{max})$. 
\begin{equation}\label{equation: f explicit}
  f_{1,1}(t) = \frac{1}{4} \left( -\frac{\rw_0(t)}{8} + \frac{\rx_0(t)^2}{12} - \frac{\ry_0(t)^2}{24} \right).
\end{equation}
For simplicity we are going to rescale this function for what follows: $f(t) := -16 f_{1,1}(t)$. Then we get:
$$
  f(t) = \frac{2}{3} X_2^{(\tau_0,\omega_0)}(t) + \frac{2}{3} X_3^{(\tau_0,\omega_0)}(t) + \frac{2}{3} X_4^{(\tau_0,\omega_0)}(t).
$$

\subsubsection{The WDVV equation}
Writing the WDVV equation of $F_0^{(\tilde E_7,G_3)}$ we get two cases: when $h_1(t)h_2(t) \equiv 0$ and $h_1(t)h_2(t) \not\equiv 0$. The first case gives system of equations that can be integrated explicitly giving $f_{1,1}(t)$ as a rational function. We know from Eq.\eqref{equation: f explicit} and the series expansion of $X_k^{(\tau_0,\omega_0)}(t)$ that this is not true. The second case is equivalent to the following system of equations:
\begin{equation}\label{eq: Example3 wdvv}
\begin{aligned}
  g_5'(t) & = \frac{16}{3} h_2(t)^2 -64 g_5(t) g_7(t),
  \\
  g_7'(t) & = 512 h_1(t) h_2(t) -32 g_7(t)^2,
  \\
  h_1'(t)& = \frac{64 h_1(t) \left(192 g_5(t) h_1(t)-g_7(t) h_2(t)\right)}{h_2(t)},
  \\
  h_2'(t)& = 64(192 g_5(t) h_1(t) -g_7(t) h_2(t)).
\end{aligned}
\end{equation}
and also:
\begin{equation}\label{eq: Example3 comp cond}
\begin{aligned}
  g_1(t)  &= \left(\frac{h_1(t)}{h_2(t)} \right)^2 g_5(t), \ g_2(t)  = \frac{1}{64} \left(g_7(t) -128 \frac{h_1(t) g_5(t)}{h_2(t)}\right), \ g_3(t)  = \frac{g_7(t)}{16},
  \\
  g_4(t)  &= \frac{g_7(t)}{16} - 6 \frac{ h_1(t) g_5(t)}{h_2(t)} , \ g_6(t)  = \frac{g_7(t)}{2}, \ f(t)  = \frac{8192 g_5(t) h_1(t)-64 g_7(t) h_2(t)}{h_2(t)}.
\end{aligned}
\end{equation}
To get the system above one should consider Eq.~\eqref{eq: wdvv} given by the following quadruples:
\begin{align*}
 \left\{t_{aJ},t_{aJ},t_{b,x^2},t_{b,x^2}\right\}, \left\{t_{aJ},t_{aJ},t_{b,x^2},t_{b,xy}\right\}, \left\{t_{aJ},t_{aJ},t_{b,x^2},t_{b,y^2}\right\} \left\{t_{aJ},t_{aJ},t_{b,xy},t_{b,xy}\right\}, 
 \\
 \left\{t_{aJ},t_{aJ},t_{b,xy},t_{b,y^2}\right\}, \left\{t_{aJ},t_{aJ},t_{b,y^2},t_{b,y^2}\right\},
  \left\{t_{aJ},t_{aJ},t_{b,x^2},t_{a^2 J}\right\}, \left\{t_{aJ},t_{aJ},t_{b,xy},t_{a^2 J}\right\}.
\end{align*}
%

\subsubsection{Solving the WDVV equation}\label{section: G_3 wdvv solution}
From Eq.\eqref{eq: Example3 wdvv} we conclude that $h_1(t) = c h_2(t)$ for some non--zero constant $c$. 

We are going to use now the relation between the functions $g_5(t)$, $g_7(t)$, $f(t)$ and explicitly known functions $X_k^{(\tau_0,\omega_0)}(t)$.
Due to the oddness of the functions $g_5(t)$ and $g_7(t)$ and Eq.~\eqref{eq: Example3 comp cond} we see that there is an odd function $p(t)$, s.t. holds:
$$
  g_7(t)= \frac{1}{64}p(t) - \frac{X_3^{(\tau_0,\omega_0)}}{32} , \quad g_5(t) = \frac{1}{8192 c}p(t) + \frac{1}{12288 c} \left(X_2^{(\tau_0,\omega_0)} + X_4^{(\tau_0,\omega_0)} - 2 X_3^{(\tau_0,\omega_0)}\right).
$$
From the first two PDEs on $g_5$ and $g_7$ we get the compatibility condition:
$$
  \frac{3}{16} \left(g_5'(t)+64 g_5(t) g_7(t)\right) = \frac{1}{512 c} \left(g_7'(t)+32 g_7(t)^2\right),
$$
that gives us the expression of $p^\prime(t)$ via $p(t)$ and $X_k^{(\tau_0,\omega_0)}$:
$$
  p^\prime(t) = p(t) \left(p(t) + 2 \left(X_2^{(\tau_0,\omega_0)} - X_3^{(\tau_0,\omega_0)} + X_4^{(\tau_0,\omega_0)}\right)\right).
$$
From the PDE on $g_7(t)$ we get the expression of $h_2(t)$, that we put into the PDE of $h_2(t)$ and get by using the formula for $p(t)$ above:
$$
  3 p(t) \left(p(t) + 2 \left(X_2^{(\tau_0,\omega_0)} - X_3^{(\tau_0,\omega_0)}\right)\right) \left(p(t) + 2 \left(X_4^{(\tau_0,\omega_0)} -X_3^{(\tau_0,\omega_0)} \right)\right) = 0,
$$
from where we find the function $p(t)$ to be one of the following three:
$$
  p(t) = 0, \ p(t) = -2 \left(X_2^{(\tau_0,\omega_0)}-X_3^{(\tau_0,\omega_0)}\right), \ p(t) = 2 \left(X_3^{(\tau_0,\omega_0)} - X_4^{(\tau_0,\omega_0)}\right)
$$
giving  the different solutions:
\begin{subequations}\label{eq: Example3 wdvv solutions}
\begin{equation}\label{eq: Example3 fjr}
\begin{aligned}
  g_7(t) =  -\frac{1}{32} X_3^{(\tau_0,\omega_0)}, & \quad g_5(t) =  \frac{1}{12288 c} \left(X_2^{(\tau_0,\omega_0)}-2 X_3^{(\tau_0,\omega_0)}+X_4^{(\tau_0,\omega_0)}\right),
  \\
  &h_2(t) =  \frac{1}{128} \sqrt{-\frac{1}{c}\left(X_2^{(\tau_0,\omega_0)}-X_3^{(\tau_0,\omega_0)}\right) \left(X_3^{(\tau_0,\omega_0)}-X_4^{(\tau_0,\omega_0)}\right)}.
\end{aligned}
\end{equation}
\begin{equation}
\begin{aligned}
  g_7(t) =  -\frac{1}{32} X_2^{(\tau_0,\omega_0)}, & \quad g_5(t) =  \frac{1}{12288 c} \left(-2 X_2^{(\tau_0,\omega_0)}+X_3^{(\tau_0,\omega_0)}+X_4^{(\tau_0,\omega_0)}\right),
  \\
  &h_2(t) =  \frac{1}{128} \sqrt{\frac{1}{c} \left(X_2^{(\tau_0,\omega_0)}-X_3^{(\tau_0,\omega_0)}\right) \left(X_2^{(\tau_0,\omega_0)}-X_4^{(\tau_0,\omega_0)}\right)}.
\end{aligned}
\end{equation}
\begin{equation}\label{eq: Example3 auxilary}
\begin{aligned}
  g_7(t) =  -\frac{1}{32} X_4^{(\tau_0,\omega_0)}, & \quad g_5(t) =  \frac{1}{12288 c} \left(X_2^{(\tau_0,\omega_0)}+X_3^{(\tau_0,\omega_0)}-2 X_4^{(\tau_0,\omega_0)}\right),
  \\
  &h_2(t) =  \frac{1}{128} \sqrt{\frac{1}{c} \left(X_2^{(\tau_0,\omega_0)}-X_4^{(\tau_0,\omega_0)}\right) \left(X_3^{(\tau_0,\omega_0)}-X_4^{(\tau_0,\omega_0)}\right)}.
\end{aligned}
\end{equation}
\end{subequations}

Actually only one of them --- Eq.~\eqref{eq: Example3 fjr} is correct for the FJRW theory because $g_7(t)$ is odd by the selection rule and from the series expansions of $X_k^{(\tau_0,\omega_0)}$ we know that only $X_3^{(\tau_0,\omega_0)}$ is odd. 

At the same time it's clear that the rescaling of the variables $t_{b,y^2} \to t_{b,y^2}/c$ and $t_{b,x^2} \to c t_{b,x^2}$ preserves the pairing fixed by $F_0^{(\tilde E_7,G_3)}$ and in the new coordinates this constant $c$ doesn't appear in the potential anymore. 

Hence up to this rescaling we can set $c=1$ and the WDVV equation has the unique solution. We get:

\begin{proposition}\label{proposition: primary potential E,G_3}
Up to a scaling of the variables the primary FJRW potential reads:
\begin{equation*}
  \begin{aligned}
    F_0^{(\tilde E_7,G_3)} & = \frac{1}{2} t_{aJ}^2 t_J+\frac{1}{2} t_J^2 t_{a^2 J}+\frac{1}{32} t_J t_{b,xy}^2+\frac{1}{16} t_J t_{b,x^2} t_{b,y^2}
    \\
    & +\left(\frac{t_{b,x^2}^4}{12288}+\frac{t_{b,y^2}^4}{12288} -\frac{t_{b,xy}^4}{6144} -\frac{t_{aJ}^4}{24} -\frac{t_{b,x^2}^2 t_{b,y^2}^2}{2048}\right)  \left( X_2^{(\tau_0,\omega_0)} + X_4^{(\tau_0,\omega_0)} \right)
    \\
    &+\frac{1}{128} \left( t_{aJ} t_{b,xy} t_{b,y^2}^2 + t_{aJ} t_{b,x^2}^2 t_{b,xy} \right) \sqrt{\left(X_3^{(\tau_0,\omega_0)} -X_2^{(\tau_0,\omega_0)} \right) \left(X_3^{(\tau_0,\omega_0)} -X_4^{(\tau_0,\omega_0)} \right)}
    \\
    & -\Bigg(\frac{1}{24} t_{aJ}^4 + \frac{t_{b,x^2}^4}{6144} + \frac{1}{64} t_{aJ}^2 t_{b,xy}^2 + \frac{t_{b,xy}^4}{6144} + \frac{1}{32} t_{aJ}^2 t_{b,x^2} t_{b,y^2} + \frac{1}{512} t_{b,x^2} t_{b,xy}^2 t_{b,y^2}
    \\
    & \quad +\frac{t_{b,x^2}^2 t_{b,y^2}^2}{1024} +\frac{t_{b,y^2}^4}{6144}\Bigg) X_3^{(\tau_0,\omega_0)},
  \end{aligned}
\end{equation*}
where $X_k^{(\tau_0,\omega_0)} = X_k^{(\tau_0,\omega_0)}\left(t_{a^2 J}\right)$ are as in Section~\ref{section: SES Gmax}. 
Moreover there is an S--action of Givental, performing the scaling of the variables, s.t. $\hat S \cdot F_0^{(\tilde E_7,G_3)} \in \QQ[[\bt]]$
\end{proposition}
\begin{proof}
  The first part follows immediately from the preceding sections.
  
  From the explicit series expansions of the functions $X_a^{(\tau_0,\omega_0)}$ we have:
  $$
    g_1(t), \ \dots ,\ g_7(t) \in \QQ[[t]], \quad f_{1,1}(t) \in \QQ[[t]],
  $$
  and
  $$
    h_1(t),h_2(t) \in \sqrt{-1} \QQ[[t^2]].
  $$
  Hence we see that $F_0^{(\tilde E_7, G_3)} \not \in \QQ[[\bt]]$.
  
  Consider the rescaling $X_a^{(\tau_0,\omega_0)}(t) \to \sqrt{-1} X_a^{(\tau_0,\omega_0)}(\sqrt{-1} t)$, that can be easily realized as a scaling of the variables, preserving the cubic terms. Note also that we have the relations 
  \[
  \sqrt{-1} X_a^{(\tau_0,\omega_0)}(\sqrt{-1} t) = X_a^{(\tau_0, \omega_1)}(t) 
  \]
  for $\omega_1 := \exp(-\pi\sqrt{-1}/2) \omega_0$, that is equivalent to the rescaling discussed. We get:
  $$
    g_a(t) \to \sqrt{-1} g_a \left(\sqrt{-1} t \right) \in \QQ[[t]], \quad f_{1,1}(t) \to \sqrt{-1} f_{1,1} \left( \sqrt{-1} t \right) \in \QQ[[t]],
  $$
  because these functions are odd, and
  $$
    h_a(t) \to \sqrt{-1} h_a(\sqrt{-1}t) \in \QQ[[t^2]],
  $$
  because these functions are even.
\end{proof}

\subsubsection{CY/LG correspondence}
Note that all three solutions from Eq.~\eqref{eq: Example3 wdvv solutions} to the WDVV equation~\eqref{eq: Example3 wdvv} differ just by the permutations of the functions $X_a^{(\tau_0, \omega_0)}$. All three solutions give some genus zero primary CohFT potentials, but only one of them is indeed a FJRW--theory genus zero primary potential as we have shown above. 

Denote the genus zero primary potential of the third WDVV solution --- Eq.~\eqref{eq: Example3 auxilary} by $F^{aux}_0$.
We identify this potential with the $\A^{(\tau_1,\omega_1)}$--transformed potential of $F_0^{\PP^1_{2,2,2,2}}$. Then Lemma~\ref{lemma: Xk permutation} gives the CY/LG correspondence action.

\subsubsection{Computation of $F_0^{aux}$}
Comparing to the previously computed FJRW theories here we also make use of Lemma~\ref{lemma: double argument of XA} and Lemma~\ref{lemma: X_k scaling}. 
We get:
\begin{align*}
    4 \cdot 32 \cdot g_7(4t) & =  -\left( 4 X_4^\infty \left(4 t \right) \right)^{(\tau_0,\omega_0)} = -\frac{1}{2} \left( (2 X_3\left(2t\right) )^{(\tau_0,\omega_0)} + ( 2 X_4\left(2t\right))^{(\tau_0,\omega_0)} \right)
    \\
    &= -\frac{1}{2} \left(X_3^{(\tau_1,\omega_1)}\left(t\right) + X_4^{(\tau_1,\omega_1)}\left(t\right) \right)
\end{align*}
where $\tau_1 = 2\tau_0$ and $\omega_1 = \omega_0/\sqrt{2}$. Similarly we have:
\begin{align*}
    & 4 \cdot 12288 c \cdot g_5(4t) =  \left(4 X_2^\infty \left(4 t \right) \right)^{(\tau_0,\omega_0)} + \left( 4 X_3^\infty \left( 4 t \right) \right)^{(\tau_0,\omega_0)} - 2 \left( 4 X_4^\infty \left(4 t \right) \right)^{(\tau_0,\omega_0)} 
    \\
    & \quad\quad = 2 X_2^{(\tau_1,\omega_1)} \left(t\right) - X_3 ^{(\tau_1,\omega_1)} \left(t\right) - X_4^{(\tau_1,\omega_1)} \left(t\right),
   \\
    & 4 \cdot 128 \sqrt{c} \cdot h_2(4t) =  
    \\
    & \sqrt{\left( \left( 4 X_2^\infty (4 t) \right)^{(\tau_0,\omega_0)} - \left( 4 X_4^\infty \left(4 t \right) \right)^{(\tau_0,\omega_0)} \right) \left( \left( 4 X_3^\infty (4 t) \right)^{(\tau_0,\omega_0)} - \left( 4 X_4^\infty \left(4 t \right) \right)^{(\tau_0,\omega_0)} \right)}
    \\
    & \quad\quad = \frac{1}{2} \left( X_3^{(\tau_1,\omega_1)} \left(t\right) - X_4^{(\tau_1,\omega_1)} \left(t\right)  \right).
\end{align*}
Applying the following linear change of the variables:
\begin{align*}
  & t_J = t_0, \ t_{a^2 J}= \tau
  \\
  & t_{aJ} = \frac{1}{2} \left(t_1-t_3\right), \ t_{b,x^2}= 2 t_2 + 2 \sqrt{-1} t_4, \ t_{b,xy} = 2 \left(t_1 + t_3\right), \ t_{b,y^2}= 2 t_2 - 2\sqrt{-1} t_4.
\end{align*}
to the potential $F_0^{aux}$ we get:
\begin{align*}
  F_0^{aux}(\bt) &= \frac{1}{2} t_0^2 \tau +\frac{1}{4} t_0 \sum_{k=2}^5 t_k^2  
  -\frac{1}{64} \left(t_3^2 t_4^2 + t_1^2 t_2^2 \right) X_4^{(\tau_1,\omega_1)} \left(\frac{\tau}{4}\right) -\frac{1}{64} \left(t_1^2 t_3^2 + t_2^2 t_4^2 \right) X_2^{(\tau_1,\omega_1)} \left(\frac{\tau}{4}\right)
  \\
  &  - \frac{1}{64} \left( t_2^2 t_3^2 + t_1^2 t_4^2 \right) X_3^{(\tau_1,\omega_1)} \left(\frac{\tau}{4}\right)  -\frac{1}{4\cdot 96} \left( \sum_{k=2}^5 t_k^4 \right) \left( \sum_{l=2}^4 X_l ^{(\tau_1,\omega_1)} \left(\frac{\tau}{4}\right) \right).
\end{align*}
Applying again Lemma~\ref{lemma: X_k scaling} we have for $\tau_2 = \tau_1/4$ and $\omega_2 = 2\omega_1$:
\begin{align*}
  F_0^{aux}(\bt) &= \frac{1}{2} t_0^2 \tau +\frac{1}{4} t_0 \sum_{k=2}^5 t_k^2  
  -\frac{1}{16} \left(t_3^2 t_4^2 + t_1^2 t_2^2 \right) X_4^{(\tau_2,\omega_2)} \left(\tau\right) -\frac{1}{16} \left(t_1^2 t_3^2 + t_2^2 t_4^2 \right) X_2^{(\tau_2,\omega_2)} \left(\tau\right)
  \\
  &  - \frac{1}{16} \left( t_2^2 t_3^2 + t_1^2 t_4^2 \right) X_3^{(\tau_2,\omega_2)} \left(\tau\right)  -\frac{1}{96} \left( \sum_{k=2}^5 t_k^4 \right) \left( \sum_{l=2}^4 X_l ^{(\tau_2,\omega_2)} \left(\tau\right) \right).
\end{align*}
Therefore, for $\tau_3 = 1 + \tau_0/2$ and $\omega_3 = \sqrt{2} \omega_0 $ holds:
$$
  F_0^{(\tilde E_7, G_3)}(\tilde \bt(\bt))  = \A^{(\tau_3,\omega_3)} \cdot F_{an}^{\PP^1_{2,2,2,2}},
$$
In order to derive the equality for the potential $F_0^{\PP^1_{2,2,2,2}}$ we apply now Proposition~\ref{proposition: sl Fan to F}. We have got:

\begin{proposition}\label{proposition: G_3 explicit}
For the linear change of the variables as above holds:
\[
  F_0^{(\tilde E_7, G_3)}(\tilde \bt(\bt))  = A^{G_3} \cdot F_0^{\PP^1_{2,2,2,2}}, \quad 
  A^{G_3} := 
    \begin{pmatrix}
	\dfrac{2\sqrt{-1}+1}{2\Theta} &  \pi \Theta \left(\sqrt{-1} - \dfrac{1}{2} \right)
	\\
	\dfrac{1}{\pi \Theta} & \Theta
      \end{pmatrix}
\]
for $\Theta = \sqrt{2\pi}/\left(\Gamma(\frac{3}{4})\right)^2$.
\end{proposition}

\appendix

\section{Some formulae on the theta constants}\label{appendix A}
  The Jacobi theta constants have the following connection to the Fourier series $\psi_k(q)$, $k=2,3,4$ of Section~\ref{section: GW theory}:
  \begin{equation*}
    \left(\vartheta_2(q)\right)^4 = 2 (\psi_3(q) - \psi_4(q)), \ \left(\vartheta_3(q)\right)^4 = 2 (\psi_2(q) - \psi_4(q)), \ \left(\vartheta_4(q)\right)^4 = 2 (\psi_2(q) - \psi_3(q)).
  \end{equation*}
  Note that these equalities are not enough to express $\psi_k(q)$ via the theta constants. We also have the following double argument formulae:
  \begin{align*}
    \left( \vartheta_2(q^2) \right)^2 = & \frac{1}{2} \left( \left(\vartheta_3(q) \right)^2 - \left(\vartheta_4(q) \right)^2 \right),
    \quad
    \left( \vartheta_3(q^2) \right)^2 = \frac{1}{2} \left( \left(\vartheta_3(q) \right)^2 + \left(\vartheta_4(q) \right)^2 \right),
    \\
    & \quad \left( \vartheta_4(q^2) \right)^2 = \vartheta_3(q)\vartheta_4(q).
  \end{align*}
  Combining these formulae with the definition of the functions $X_k^\infty(q)$ we get:
  \begin{equation*}
  \begin{aligned}
    2 X_2^\infty (q^2)  &= \frac{X_3^\infty(q) \left(\vartheta_3(q)\right)^2 -X_4^\infty(q) \left(\vartheta_4(q)\right)^2}{\left(\vartheta_3(q)\right)^2-\left(\vartheta_4(q)\right)^2},
    \\
    2 X_3^\infty (q^2)  &= \frac{X_3^\infty(q) \left(\vartheta_3(q)\right)^2 +X_4^\infty(q) \left(\vartheta_4(q)\right)^2}{\left(\vartheta_3(q)\right)^2+\left(\vartheta_4(q)\right)^2},
    \\
    2 X_4^\infty (q^2)  &= \frac{1}{2}\left(X_3^\infty(q) + X_4^\infty(q)\right).
  \end{aligned}
  \end{equation*}

  The following lemma is only applicable to the scaling of $\tau$ by $2$ and uses double argument formulae of the theta constants.
  \begin{lemma}\label{lemma: double argument of XA}
    For any $A \in \mathrm{SL}(2,\CC)$ we have the following equalities:
    \begin{align*}
      & \left(2 X_2(2 \tau)\right)^A = \frac{X_3^A(\tau) T_3^A(\tau) - X_4^A(\tau) T_4^A(\tau)}{ T_3^A(\tau) - T_4^A(\tau)},
      \\
      & \left(2 X_3(2 \tau)\right)^A = \frac{X_3^A(\tau) T_3^A(\tau) + X_4^A(\tau) T_4^A(\tau)}{ T_3^A(\tau) + T_4^A(\tau)},
      \\
      & \left(2 X_4(2 \tau)\right)^A = \frac{1}{2}\left(X_3^A(\tau) + X_4^A(\tau) \right),
    \end{align*}
    where
    $$
      T_k^A(\tau) := \frac{1}{c \tau + d} \left( \vartheta_k \left(\frac{a \tau + b}{c \tau + d} \right) \right)^2, \quad k=2,3,4.
    $$
  \end{lemma}
  \begin{proof}
    First of all note that we can not apply $A$ to the function $X_k^\infty(2 \tau)$ because the latter one doesn't solve the Halphen's system. Let's apply it to $2 X_k^\infty(2 \tau)$. We only do it in one example, while all the other are similar. Let:
    $$      
      A = \begin{pmatrix}
           a & b \\ c & d
          \end{pmatrix} \in \SL2C, \quad \text{ and } \quad \tau' := \frac{a \tau + b}{c \tau + d}.
    $$
    Using the double argument formula for $X_2^\infty$ above we have:
    \begin{align*}
      \left( 2 X_2^\infty (2 \tau) \right)^A &= \frac{1}{(c \tau + d)^2} \cdot 2 X_2^\infty \left( 2 \frac{a \tau + b}{c \tau + d} \right) + \frac{c}{c \tau + d},
      \\
      &= \frac{1}{(c \tau + d)^2} \frac{X_3^\infty (\tau') \vartheta_3^2 (\tau') - X_4^\infty (\tau') \vartheta_4^2(\tau')}{\vartheta_3^2(\tau') - \vartheta_4^2(\tau')} + \frac{c}{c \tau + d}
      \\
      &= \frac{\left[ X_3^\infty (\tau')  + c(c \tau + d) \right]\vartheta_3^2 (\tau')  - \left[ X_4^\infty (\tau') + c(c \tau + d) \right]\vartheta_4^2(\tau') }{(c \tau + d)^2(\vartheta_3^2(\tau') - \vartheta_4^2(\tau'))}.
    \end{align*}
    The other two cases are treated in the same way.
  \end{proof}
  For a more general scaling we have.
  \begin{lemma}\label{lemma: X_k scaling}
    For any $\tau_0 \in \HH$, $\omega_0 \in \CC^*$ and $k \in \QQ_{> 0}$ holds:
    $$
       \left( k X_a^\infty(k \tau) \right) ^{(\tau_0,\omega_0)} = \left( X_a^\infty(\tau) \right)^{(\tau_1,\omega_1)}, \quad 2 \le a \le 4,
    $$
    where $\tau_1 = k \tau_0, \ \omega_1 = \omega_0 / \sqrt{k}$.
  \end{lemma}
  \begin{proof}
    First of all note that the formula given makes sense. Namely, the triple of functions $k X_a^\infty(k\tau)$ is solution of the Halphen's system too. The rest follows from the following equalities.
    \begin{align*}
      \left( k X_a^\infty(k \tau) \right) & ^{(\tau_0,\omega_0)} = \A^{(\tau_0,\omega_0)} \cdot \left( k X_a^\infty(k \tau) \right) 
      \\
      & = k \frac{(2 \omega_0 \Im(\tau_0))^2}{(\sqrt{-1}\tau + 2 \omega_0^2 \Im(\tau_0))^2} X_a^\infty \left( k \cdot \frac{\sqrt{-1}\tau \bar\tau_0 + \tau_0 \cdot 2 \omega_0^2 \Im(\tau_0)}{\sqrt{-1}\tau + 2 \omega_0^2 \Im(\tau_0)} \right) 
      \\
      & \quad - \frac{1}{\tau - 2 \sqrt{-1} \omega_0^2 \Im(\tau_0)} 
      = \A^{(\tau_1,\omega_1)} \cdot \left( X_a^\infty(\tau) \right).
    \end{align*}
  \end{proof}

  \begin{lemma}\label{lemma: Xk permutation}
    For any $\tau_0 \in \HH$ and $\omega_0 \in \CC^*$ holds:
    $$
      X_2^{(\tau_0,\omega_0)}(t) = X_2^{(\tau_1,\omega_0)}(t), \ X_3^{(\tau_0,\omega_0)}(t) = X_4^{(\tau_1,\omega_0)}(t), \ X_4^{(\tau_0,\omega_0)}(t) = X_3^{(\tau_1,\omega_0)}(t),
    $$
    for $\tau_1 := \tau_0 + 1$.
  \end{lemma}
  \begin{proof}
    This follows immediately from the identities $X_2^\infty(t+1) = X_2^\infty(t)$, $X_3^\infty(t+1) = X_4^\infty(t)$, $X_4^\infty(t+1) = X_3^\infty(t)$ and the definition of the the $A^{(\tau_0, \omega_0)}$ --action.
  \end{proof}

\section{Gromov--Witten potential of $\PP^1_{4,4,2}$}\label{section: appendix GW}
In order to shorten the formulae let $t_k := t_{1,k}$ for $1 \le k \le 3$, $t_l := t_{2,l-3}$ for $4 \le l \le 6$, $t_7 := t_{3,1}$.
Let 
$x = x(q)$, $y = y(q)$, $z = z(q)$, $w = w(q)$ be as in Section~\ref{section: GW theory}.
The following expression for the genus zero GW potential of $\PP^1_{4,4,2}$ was published in \cite{BP}.

\begingroup
\everymath{\scriptstyle}
\tiny
\begin{align*}
  & F_0^{\PP^1_{4,4,2}} = -\frac{\left(x^6-5 x^4 y^2-5 x^2 y^4+y^6\right) }{4128768}\left(t_3^8+t_6^8\right)+\frac{x y \left(x^4+14 x^2 y^2+y^4\right) }{294912}t_3^2 t_6^2 \left(t_3^4+t_6^4\right)+\frac{z \left(8 x^4+8 y^4+19 z^4\right)}{294912} t_6^3 t_7 t_3^3
  \\
  &+\frac{x \left(x^2+y^2\right)^2 }{73728}\left(t_2 t_3^6+t_5 t_6^6\right)+\frac{y \left(x^2+y^2\right)^2 }{73728}\left(t_3^6 t_5+t_2 t_6^6\right)+\frac{5 x^2 y^2 \left(x^2+y^2\right) }{73728}t_6^4 t_3^4-\frac{\left(x^4-6 x^2 y^2+y^4\right) }{30720}\left(t_1 t_3^5+t_4 t_6^5\right)
  \\
  &-\frac{\left(x^4-3 x^2 y^2\right) }{3072}\left(t_2^2 t_3^4+t_5^2 t_6^4\right)+\frac{\left(3 x^2 y^2-y^4\right) }{3072}\left(t_3^4 t_5^2+t_2^2 t_6^4\right)+\frac{x y z\left(x^2+y^2\right) }{6144}t_3 t_6 \left(t_3^4+t_6^4\right) t_7+\frac{x^2 y \left(x^2+4 y^2\right)}{6144}t_3^2 t_6^2 \left(t_2 t_3^2+t_5 t_6^2\right)
  \\
  &+\frac{x y^2 \left(4 x^2+y^2\right) }{6144}t_3^2 t_6^2 \left(t_3^2 t_5+t_2 t_6^2\right)+\frac{x y \left(x^2+y^2\right) }{1536}\left(t_3^2 t_6^2 \left(t_1 t_3+t_4 t_6\right)+t_2 t_5 \left(t_3^4+t_6^4\right)\right)+\frac{x^2 y^2 }{1536}t_3 t_6 \left(t_3^3 t_4+t_1 t_6^3\right)
  \\
  &+\frac{x z\left(x^2+7 y^2\right) }{1536}t_3 t_6 t_7\left(t_3^2 t_5+t_2 t_6^2\right)+\frac{y z\left(7 x^2+y^2\right)}{1536}t_3 t_6 t_7\left(t_2 t_3^2+t_5 t_6^2\right)+\frac{x y \left(x^2+y^2\right)}{512} t_3^2 t_6^2\left(t_2^2+t_5^2\right) +\frac{x^2 y^2 }{384} \left(t_3^4+t_6^4\right) t_7^2
  \\
  &+\frac{x \left(x^2+y^2\right)}{384} \left(t_1 t_2 t_3^3+t_4 t_5 t_6^3\right)+\frac{y \left(x^2+y^2\right)}{384} \left(t_1 t_3^3 t_5+t_2 t_4 t_6^3\right)+\frac{\left(x^2+y^2\right) z}{384} t_7\left(t_3^3 t_4+t_1 t_6^3\right) +\frac{x^3}{384}\left(t_2^3 t_3^2+t_5^3 t_6^2\right)
  \\
  &+\frac{y^3 }{384} \left(t_3^2 t_5^3+t_2^3 t_6^2\right)-\frac{\left(3 w-x^2+2 y^2\right)}{384} \left(t_2^4+t_5^4\right)+\frac{x y^2 }{128} t_2 t_5 \left(t_3^2 t_5+t_2 t_6^2\right)+\frac{x^2 y }{128} t_2 t_5 \left(t_2 t_3^2+t_5 t_6^2\right)+\frac{x^2 y^2}{128} t_2 t_5 t_6^2 t_3^2
  \\
  &+\frac{x y \left(x^2+y^2\right)}{128} t_6^2 t_7^2 t_3^2+\frac{\left(2 x^2-y^2-3 w\right)}{96} t_7^4+\frac{x y^2}{64} t_3 t_6 \left(t_2 t_3 t_4+t_1 t_5 t_6\right)+\frac{x^2 y}{64} t_3 t_6 \left(t_3 t_4 t_5+t_1 t_2 t_6\right)
  \\
  &+\frac{x y z}{192} t_3 t_6 t_7 \left(3 t_2^2+3 t_1 t_3+3 t_5^2+3 t_4 t_6+4 t_7^2\right)+\frac{z\left(x^2+y^2\right)}{64} t_2 t_5 t_6 t_7 t_3-\frac{\left(w-x^2\right) }{64} \left(2 t_5^2 t_7^2+t_2^2 t_5^2+2 t_2^2 t_7^2\right)
  \\
  &-\frac{\left(2 w-x^2+y^2\right)}{64} \left(t_1^2 t_3^2+t_4^2 t_6^2\right)+\frac{x y^2}{32} \left(t_2 t_7^2 t_3^2+t_5 t_6^2 t_7^2\right)+\frac{x^2 y }{32} \left(t_5 t_7^2 t_3^2+t_2 t_6^2 t_7^2\right)+\frac{x y }{32} \left(2t_1 t_2 t_5 t_3+t_1^2 t_6^2+t_4^2 t_3^2\right)
  \\
  &-\frac{w}{32} \left(t_4 t_5^2 t_6+t_1 t_2^2 t_3\right)+\frac{\left(x^2-y^2-w\right)}{32} \left(t_1 t_5^2 t_3+ t_2^2 t_4 t_6\right)-\frac{\left(w-x^2\right)}{16} \left(t_1 t_7^2 t_3+ t_1 t_4 t_6 t_3+ t_4 t_6 t_7^2\right)+\frac{x y }{16} t_2 t_5\left(t_4t_6+2 t_7^2\right)
  \\
  &+\frac{x z}{16} t_7\left(t_2 t_4 t_3+t_1 t_5 t_6\right)+\frac{y z }{16} t_7 \left(t_3t_4 t_5+t_1 t_2 t_6\right)+\frac{x}{8} \left(t_1^2 t_2+t_4^2 t_5\right)+\frac{y}{8} \left(t_2 t_4^2+ t_1^2 t_5\right)+\frac{z}{4} t_1 t_4 t_7
  \\
  &+\frac{1}{8} t_0 \left(t_2^2+t_5^2+2 t_7^2+2 t_1 t_3+2 t_4 t_6\right)+\frac{1}{2} t_0^2 t_{-1}.
\end{align*}
\endgroup

\end{document}